\documentclass[11pt]{article}

\usepackage{verbatim}
\usepackage{amssymb}
\usepackage{amsbsy}
\usepackage{amscd}
\usepackage{amsmath}
\usepackage{amsthm}
\usepackage[mathscr]{euscript}
\usepackage{enumerate,esint}
\usepackage{xcolor}
\definecolor{darkgreen}{rgb}{0,0.75,0}
\definecolor{darkred}{rgb}{0.75,0,0}
\definecolor{darkmagenta}{rgb}{0.5,0,0.5}
\usepackage[colorlinks,citecolor=darkgreen,linkcolor=darkred,urlcolor=darkmagenta]{hyperref}

\allowdisplaybreaks[4]
\newcommand{\on}[1]{\operatorname{#1}}
\newcommand{\abs}[1]{\left\lvert#1\right\rvert}

\newcommand{\one}{\mathbf{1}}
\newcommand{\supp}{\operatorname{supp}}
\newcommand{\diam}{\operatorname{diam}}
\newcommand{\Gam}{\Gamma}

\newcommand{\sL}{\mathcal{L}}
\newcommand{\sH}{\mathcal{H}}

\newcommand{\bR}{\mathbb{R}}
\newcommand{\bN}{\mathbb{N}}

\newcommand{\bH}{\mathbb{H}}
\newcommand{\R}{\mathbb{R}}

\title{\fontsize{17.28}{20.74}\selectfont\bfseries Martingale and analytic dimensions}
\author{Sylvester Eriksson-Bique \and Mathav Murugan}
\date{}

\DeclareMathOperator{\dom}{dom}
\DeclareMathOperator{\gap}{gap}
\DeclareMathOperator{\len}{len}
\DeclareMathOperator{\Lip}{Lip}
\DeclareMathOperator{\LIP}{LIP}
\DeclareMathOperator{\PI}{PI}

\newcommand{\N}{\mathbb{N}}
\newcommand{\restr}[2]{\left.#1\right|_{#2}}
\newtheorem{theorem}{Theorem}[section]
\newtheorem{lemma}[theorem]{Lemma}
\newtheorem{proposition}[theorem]{Proposition}

\theoremstyle{definition}
\newtheorem{definition}[theorem]{Definition}
\theoremstyle{definition}
\newtheorem{remark}[theorem]{Remark}
\newtheorem{example}[theorem]{Example}
\newtheorem{question}[theorem]{Question}
\newtheorem{conjecture}[theorem]{Conjecture}

\numberwithin{equation}{section}

\begin{document}
\maketitle
\vspace{-0.5cm}
\begingroup
\renewcommand{\thefootnote}{}
\endgroup
\setcounter{footnote}{0}

\begin{abstract}
	These notes introduce the martingale dimension of a strongly local
	regular Dirichlet form and the analytic dimension of a Lipschitz
	differentiability space. We explain why these dimensions coincide
	under two-sided Gaussian heat kernel estimates and discuss their
	relationships with other notions of dimension. We also introduce the
	energy image density property, its recent resolution, and its
	applications to martingale dimension. Finally, we show that spaces
	carrying diffusions with two-sided sub-Gaussian heat kernel estimates
	of walk dimension greater than two are not Lipschitz differentiability
	spaces.
\end{abstract}

Strongly local regular Dirichlet forms provide an analytic framework for
describing symmetric diffusions through their associated Markov semigroups.
Lipschitz differentiability spaces, on the other hand, arise from extending
Rademacher's theorem on the almost everywhere differentiability of Lipschitz
functions to metric measure spaces. These different motivations have led
the two theories to develop largely independently.

Despite their different origins, martingale dimension for strongly local
regular Dirichlet spaces and analytic dimension for Lipschitz
differentiability spaces can both be interpreted as notions of
``tangent-space dimension'' in nonsmooth settings. The energy image density
property, which originated in Malliavin calculus, and related
absolute-continuity results provide a connection between the two theories
and play an important role in understanding these dimensions.

The aim of these notes is to introduce these two notions of dimension,
explain their connections, and highlight their similarities and differences.
Throughout, we discuss open problems that illustrate how incomplete our
understanding of these dimensions remains. The first three sections largely follow the three lectures delivered by
the second-named author at the 18th MSJ-SI in Fukuoka. The final
section presents a new result showing that the underlying metric measure
spaces of a broad class of Dirichlet spaces arising in the study of
diffusions on fractals are not Lipschitz differentiability spaces.

\section{Lecture 1: Definitions of martingale and analytic dimensions}
 
In this lecture, we introduce the notions of martingale and analytic
dimensions, using Euclidean space and the Heisenberg group as our main
examples. These examples illustrate how both notions recover the
dimension of the tangent space in Euclidean space and of the horizontal
tangent space in the Heisenberg group.

\subsection{Strongly local Dirichlet forms}
The notion of martingale dimension is formulated in terms of the diffusion
associated with a regular strongly local Dirichlet form, whose definition we
recall below.

	Let $X$ be a locally compact, separable, metrizable topological
	space, and let $m$ be a Radon measure on $X$ with full support. A pair $(\mathcal{E},\mathcal{F})$
	is called a \emph{Dirichlet form} on $L^{2}(X,m)$ if $\mathcal{F}$ is a
	dense linear subspace of $L^{2}(X,m)$ and
	$\mathcal{E}\colon\mathcal{F}\times\mathcal{F}\longrightarrow\mathbb{R}$
	is a symmetric, non-negative definite bilinear form that is closed and
	Markovian. Here, closedness means that $\mathcal{F}$ is a Hilbert space
	with respect to the inner product
	\[
	\mathcal{E}_{1}(f,g)
	:=
	\mathcal{E}(f,g)+\langle f,g\rangle_{L^{2}(X,m)},
	\qquad f,g\in\mathcal{F},
	\]
	and the Markov property means that, for every $f\in\mathcal{F}$,
	\[
	f^{+}\wedge 1\in\mathcal{F}
	\quad\text{and}\quad
	\mathcal{E}(f^{+}\wedge 1,f^{+}\wedge 1)
	\leq \mathcal{E}(f,f).
	\]
	
	The Dirichlet form $(\mathcal{E},\mathcal{F})$ is called \emph{regular} if
	$\mathcal{F}\cap C_{\mathrm{c}}(X)$ is dense both in
	$(\mathcal{F},\mathcal{E}_{1})$ and in
	$(C_{\mathrm{c}}(X),\|\cdot\|_{\mathrm{sup}})$.	It is called \emph{strongly local} if
	\[
	\mathcal{E}(f,g)=0
	\]
	whenever $f,g\in\mathcal{F}$, $\supp_{m}[g]$ is compact, and there exists
	$a\in\mathbb{R}$ such that
	\[
	\supp_{m}[f-a\one_{X}]\cap\supp_{m}[g]=\emptyset.
	\]
	Here, $\supp_m[f]$ denotes the support of the measure $f\cdot m$.

 The \emph{energy measure} $\Gamma(f,f)$ of $f\in\mathcal{F}$
 associated with the strongly local regular form $(\mathcal{E},\mathcal{F})$ is defined,
 first for $f\in\mathcal{F}\cap L^{\infty}(X,m)$ as the unique positive
 Borel measure on $X$ such that
 \begin{equation}\label{e:EnergyMeas}
 	\int_{X} g \, d\Gamma(f,f)= \mathcal{E}(f,fg)-\frac{1}{2}\mathcal{E}(f^{2},g) \qquad \textrm{for all $g \in \mathcal{F}\cap C_{\mathrm{c}}(X)$,}
 \end{equation}
 and then by
 $\Gamma(f,f)(A):=\lim_{n\to\infty}\Gamma\bigl((-n)\vee(f\wedge n),(-n)\vee(f\wedge n)\bigr)(A)$
 for each Borel subset $A$ of $X$ for general $f\in\mathcal{F}$. The signed measure $\Gamma(f,g)$ for $f,g \in \mathcal{F}$ is defined by polarization.

 We say that a positive $\sigma$-finite Borel measure $\nu$ is a \emph{minimal energy dominant measure} if it satisfies the following:
 \begin{enumerate}[(a)]
 	\item (Energy dominance) $\Gam(f,f) \ll \nu$ for every $f \in \mathcal{F}$, where $\Gam(f,f)$ is the energy measure of $f$.
 	\item (Minimality) If $\tilde{\nu}$ satisfies (a), then $\nu \ll \tilde{\nu}$.
 \end{enumerate}
 Note that (b) implies that any two minimal energy dominant measures are mutually absolutely continuous.
  Every strongly local, regular Dirichlet form has a minimal energy dominant measure   \cite[Lemma 2.2]{Nak}, \cite[Lemma 2.2]{hinoenergymeasures}.   Next, we recall the   definition of martingale dimension from \cite{hinoenergymeasures}.
 
 \subsection{Martingale dimension}

	Let $(\mathcal{E},\mathcal{F})$ be a strongly local regular Dirichlet form on
	$L^2(X,m)$. Denote its energy measures by $\Gamma(\cdot,\cdot)$,
	and let $\nu$ be a minimal energy dominant measure.
 	 The \emph{pointwise index} of $(\mathcal{E},\mathcal{F})$ is a measurable function
		$p_H:X\to\bN\cup\{0,\infty\}$ satisfying the following conditions:
		\begin{enumerate}[(a)]
			\item For every $N\in\bN$ and $f_1,\ldots,f_N\in\mathcal{F}$,
			\[
			\on{rank}
			\left(\frac{d\Gamma(f_i,f_j)}{d\nu}(x)\right)_{1\leq i,j\leq N}
			\leq p_H(x)
			\]
			for $\nu$-almost every $x\in X$.
			\item If a measurable function
			$p_H':X\to\bN\cup\{0,\infty\}$ satisfies \textup{(a)} in place
			of $p_H$, then $p_H\leq p_H'$ $\nu$-almost everywhere.
		\end{enumerate}
		These conditions determine $p_H$ uniquely up to $\nu$-null sets.
		  The \emph{martingale dimension} of $(\mathcal{E},\mathcal{F})$ is
		\[
		\nu\text{-}\on{esssup}_{x\in X}p_H(x).
		\]

The notion of martingale dimension and its terminology originate in
the work of Motoo and Watanabe \cite{MW} on stochastic integral
representations of martingale additive functionals of Markov processes (see also  \cite{DV}). Roughly
speaking, the martingale dimension of the associated diffusion is
the smallest number $k$ for which there is a fixed family of $k$
martingale additive functionals of finite energy such that every
martingale additive functional of finite energy admits a stochastic
integral representation with respect to this family. If no finite
family suffices, the martingale dimension is infinite.  
The equivalence between this probabilistic definition and the
analytic definition above is due to Hino
\cite[Theorem 3.4]{hinoenergymeasures}, building on earlier work
of Kusuoka \cite{Kus89,Kus93}.

\subsection{Examples: Euclidean space and the Heisenberg group}
Let $\mathcal{L}_n$ denote Lebesgue measure on $\mathbb{R}^n$.
The heat semigroup
\[
P_tf(x)=
\int_{\mathbb{R}^n}
\frac{1}{(4\pi t)^{n/2}}
\exp\left(-\frac{|x-y|^2}{4t}\right)f(y)\,dy,
\qquad t>0,
\]
is associated with the strongly local regular Dirichlet form
\[
\mathcal{F}=W^{1,2}(\mathbb{R}^n),
\qquad
\mathcal{E}(f,g)=
\int_{\mathbb{R}^n}\nabla f\cdot\nabla g\,d\mathcal{L}_n.
\]
Here the gradients are understood in the weak sense.
With this normalization,
\[
P_tf(x)=\mathbb{E}_x[f(B_{2t})],
\]
where $(B_t)_{t\geq0}$ is standard Brownian motion.

The energy measures are given by
\[
\Gamma(f,g)=
(\nabla f\cdot\nabla g)\,\mathcal{L}_n,
\qquad f,g\in\mathcal{F}.
\]
In particular, $\mathcal{L}_n$ is energy dominant. To see that it
is minimal, fix a ball $B$ and choose $\chi\in C_c^\infty(\mathbb{R}^n)$
equal to one on a neighbourhood of $\overline{B}$. The functions
$f_i(x)=\chi(x)x_i$, $1\leq i\leq n$, belong to $\mathcal{F}$ and satisfy
$\nabla f_i=e_i$ on $B$. Consequently,
\[
\Gamma(f_i,f_i)|_B=\mathcal{L}_n|_B.
\]
Every energy dominant measure therefore dominates
$\mathcal{L}_n|_B$. Taking a countable collection of balls covering
$\mathbb{R}^n$ shows that $\nu=\mathcal{L}_n$ is minimal energy dominant.

For any $f_1,\ldots,f_N\in\mathcal{F}$, the matrix of energy densities
is the Gram matrix of their gradients. Thus, almost everywhere,
\begin{align*}
	\on{rank}
	\left(\frac{d\Gamma(f_i,f_j)}{d\nu}(x)\right)_{i,j=1}^N
	&=
	\on{rank}
	\left(\nabla f_i(x)\cdot\nabla f_j(x)\right)_{i,j=1}^N\\
	&=
	\dim\on{span}\{\nabla f_1(x),\ldots,\nabla f_N(x)\}
	\leq n.
\end{align*}
It follows that $p_H\leq n$ almost everywhere. Conversely, the
cutoff coordinate functions above have the identity matrix as
their matrix of energy densities on $B$, so $p_H\geq n$ almost
everywhere on $B$. A countable covering by balls therefore gives
\[
p_H(x)=n
\qquad\text{for $\mathcal{L}_n$-almost every $x\in\mathbb{R}^n$.}
\]
Hence the martingale dimension is $n$, recovering the dimension
of the Euclidean tangent space.

 The Heisenberg group provides an example in which the martingale
 dimension differs from the dimension of the underlying smooth manifold.
 We first describe its group structure and the associated Brownian motion.
 
 The \emph{Heisenberg group} $\mathbb{H}$ is $\mathbb{R}^3$ equipped
 with the multiplication
 \[
 (x,y,z)\odot(x',y',z')
 =
 \left(x+x',\,y+y',\,z+z'+\frac12(xy'-yx')\right).
 \]
 It is a non-commutative group whose identity is $(0,0,0)$, and
 \[
 (x,y,z)^{-1}=(-x,-y,-z).
 \]
 Lebesgue measure
 $\mathcal{L}_3$ is invariant under both left and right translations
 and hence is a Haar measure on $\mathbb{H}$.
 
 The left-invariant vector fields obtained by translating the
 coordinate directions $\partial_x$ and $\partial_y$ at the identity are
 \[
 V_1=\partial_x-\frac{y}{2}\partial_z,
 \qquad
 V_2=\partial_y+\frac{x}{2}\partial_z.
 \]
 Their span at each point is called the \emph{horizontal tangent space}.
 We equip these two-dimensional spaces with the inner product for which
 $V_1,V_2$ are orthonormal. Although they span only two directions,
 their Lie bracket supplies the third:
 \[
 [V_1,V_2]=\partial_z.
 \]
 
 Let $(B_t^{(1)},B_t^{(2)})_{t\geq0}$ be standard planar Brownian
 motion started at the origin, and define
 \[
 A_t=\frac12\int_0^t
 \left(B_s^{(1)}\,dB_s^{(2)}-B_s^{(2)}\,dB_s^{(1)}\right),
 \qquad
 X_t=(B_t^{(1)},B_t^{(2)},A_t).
 \]
 The process $A_t$ is called \emph{L\'evy's stochastic area}.
 To explain the terminology, consider the one-form
 \[
 \alpha=\frac12(x\,dy-y\,dx),
 \qquad d\alpha=dx\wedge dy.
 \]
 For a smooth closed planar curve, Stokes' theorem identifies the
 integral of $\alpha$ with its signed enclosed area. The stochastic integral defining $A_t$ is the
 corresponding area functional for a Brownian path. Its It\^o and
 Stratonovich interpretations agree, since
 $[B^{(1)},B^{(2)}]_t=0$.
 The process $X_t$ solves the Stratonovich equation
 \[
 dX_t=V_1(X_t)\circ dB_t^{(1)}
 +V_2(X_t)\circ dB_t^{(2)},
 \qquad X_0=(0,0,0).
 \]
 Thus it is driven by independent Brownian noises in the two
 orthonormal horizontal directions. It is called
 \emph{horizontal Brownian motion on the Heisenberg group},
 or simply Brownian motion on $\mathbb{H}$ in what follows.
 
 The group law also gives a direct analogue of the increment
 properties of Euclidean Brownian motion. For $0\leq s<t$,
 \[
 X_s^{-1}\odot X_t
 =
 \left(B_t^{(1)}-B_s^{(1)},\,
 B_t^{(2)}-B_s^{(2)},\,A_{s,t}\right),
 \]
 where
 \[
 A_{s,t}
 =
 \frac12\int_s^t
 \left(
 (B_u^{(1)}-B_s^{(1)})\,dB_u^{(2)}
 -
 (B_u^{(2)}-B_s^{(2)})\,dB_u^{(1)}
 \right).
 \]
 This increment depends only on the Brownian increments after time
 $s$. Consequently, it is independent of the past and has the same
 law as $X_{t-s}$. Thus $X_t$ has continuous paths and stationary,
 independent increments with respect to the group operation. These properties
 parallel those of Euclidean Brownian motion and justify the
 terminology ``Brownian motion on the Heisenberg group.''

 By It\^o's formula, the generator is
 \[
 L=\frac12(V_1^2+V_2^2).
 \]
 More explicitly, for every $f\in C_c^\infty(\mathbb{R}^3)$,
 \[
 f(X_t)-f(X_0)-\int_0^t Lf(X_s)\,ds
 \]
 is a martingale. The associated strongly local regular
 Dirichlet form is the closure on $L^2(\mathbb{H},\mathcal{L}_3)$ of
 \[
 \mathcal{E}(f,g)
 =
 \frac12\int_{\mathbb{H}}
 \left(V_1f\,V_1g+V_2f\,V_2g\right)\,d\mathcal{L}_3,
 \qquad f,g\in C_c^\infty(\mathbb{R}^3).
 \]
 Denote its domain by $\mathcal{F}$ and write
 \[
 \nabla_H f=(V_1f,V_2f)
 \]
 for the horizontal gradient. For $f\in\mathcal{F}$, these
 derivatives are understood in the weak sense. The energy measures
 satisfy
 \[
 \Gamma(f,g)
 =
 \frac12\,\nabla_H f\cdot\nabla_H g\,\mathcal{L}_3,
 \qquad f,g\in\mathcal{F}.
 \]
 
 In particular, $\mathcal{L}_3$ is energy dominant. As in the
 Euclidean example, cutoff coordinate functions show that it is
 minimal. Indeed, fix a bounded Euclidean ball $B$ and choose
 $\chi\in C_c^\infty(\mathbb{R}^3)$ equal to one on a neighbourhood
 of $\overline{B}$. The functions
 \[
 f_1(x,y,z)=\chi(x,y,z)x,
 \qquad
 f_2(x,y,z)=\chi(x,y,z)y
 \]
 satisfy
 \[
 \nabla_H f_1=(1,0),
 \qquad
 \nabla_H f_2=(0,1)
 \quad\text{on }B.
 \]
 In particular, $\Gamma(f_1,f_1)|_B=\frac12\mathcal{L}_3|_B$.
 A countable covering by such balls shows that every energy
 dominant measure dominates $\mathcal{L}_3$.
 
 Taking $\nu=\mathcal{L}_3$, for any $f_1,\ldots,f_N\in\mathcal{F}$
 we obtain, almost everywhere,
 \begin{align*}
 	\on{rank}
 	\left(\frac{d\Gamma(f_i,f_j)}{d\nu}(q)\right)_{i,j=1}^N
 	&=
 	\on{rank}
 	\left(\frac12\nabla_H f_i(q)\cdot\nabla_H f_j(q)\right)_{i,j=1}^N\\
 	&=
 	\dim\on{span}\{\nabla_H f_i(q):1\leq i\leq N\}
 	\leq2.
 \end{align*}
 Hence $p_H\leq2$ almost everywhere. For the two cutoff coordinate
 functions above, the matrix of energy densities on $B$ is
 $\frac12 I_2$, so $p_H\geq2$ almost everywhere on $B$. Therefore
 \[
 p_H(q)=2
 \qquad\text{for $\mathcal{L}_3$-almost every $q\in\mathbb{H}$,}
 \]
 and the martingale dimension is two. It thus recovers the
 dimension of the horizontal tangent space, even though
 $\mathbb{H}$ is a three-dimensional smooth manifold.

\subsection{Lipschitz differentiability spaces}
Rademacher's theorem states that every Lipschitz function
$f:\mathbb{R}^n\to\mathbb{R}$ is differentiable at Lebesgue-almost
every point. Lipschitz differentiability spaces abstract this phenomenon
to metric measure spaces: up to a null set, the space is covered by
countably many measurable sets, each equipped with a finite-dimensional
Lipschitz coordinate map with respect to which every Lipschitz function
is differentiable almost everywhere. We recall the definition below
and refer to \cite{Ch,Bate,bate2024fragment,BS13,Kei-diff,LiBate,EBchar}
for background and further discussion. Here metric measure spaces are complete, locally compact and separable with a Borel measure of full support that is finite and positive on balls.
 
 	Let $(X,d,\mu)$ be a metric measure space. Let $U\subset X$ be a
 	Borel set and let $\varphi:X\to\mathbb{R}^n$ be Lipschitz, where $n \ge 1$. We say
 	that a Lipschitz function $f:X\to\mathbb{R}$ is
 	\emph{differentiable with respect to $(U,\varphi)$ at $x\in U$}
 	if there exists a unique linear map
 	$df_x:\mathbb{R}^n\to\mathbb{R}$ such that
 	\[
 	\lim_{\substack{y\to x\\y\neq x}}
 	\frac{|f(y)-f(x)-df_x(\varphi(y)-\varphi(x))|}
 	{d(x,y)}
 	=0.
 	\]
 	The linear map $df_x$ is called the \emph{differential} of $f$
 	at $x$ with respect to $(U,\varphi)$.
 	
 	A pair $(U,\varphi)$ is called a \emph{Lipschitz differentiability
 		chart of dimension $n$} if $\mu(U)>0$ and every Lipschitz function
 	$f:X\to\mathbb{R}$ is differentiable with respect to $(U,\varphi)$
 	at $\mu$-almost every $x\in U$. The exceptional $\mu$-null set may depend on $f$.
 	
 	The space $(X,d,\mu)$ is called a \emph{Lipschitz differentiability
 		space} if it admits a countable collection of Lipschitz
 	differentiability charts $\{(U_i,\varphi_i)\}_i$, of dimensions
 	$n_i\geq1$, such that
 	\[
 	\mu\left(X\setminus\bigcup_i U_i\right)=0.
 	\]
 	Such a collection is called a \emph{measurable differentiable
 		structure}.
 	
 The dimensions of overlapping charts agree $\mu$-almost everywhere
 on their overlap. Consequently, the chart dimensions determine a
 measurable function $d_C:X\to\bN$, uniquely up to $\mu$-null sets,
 such that
 \[
 d_C(x)=n_i
 \qquad\text{for $\mu$-almost every $x\in U_i$, for each $i$.}
 \]
 Moreover, $d_C$ is independent, up to $\mu$-null sets, of the chosen
 measurable differentiable structure. We call $d_C$ the
 \emph{pointwise analytic dimension}.
 	The \emph{analytic dimension} (or \emph{Lipschitz differentiability
 		dimension}) of $(X,d,\mu)$ is
 	\[
 	\mu\text{-}\on{esssup}_{x\in X}d_C(x).
 	\]
 	Equivalently, it is the smallest $N\in\bN$ such that every
 	Lipschitz differentiability chart has dimension at most $N$,
 	with value $\infty$ if no such finite bound exists.
 
 \subsection{Examples of Lipschitz differentiability spaces}
 If $X=\mathbb{R}^n$ is equipped with the Euclidean distance and
 $\varphi:\mathbb{R}^n\to\mathbb{R}^n$ is the identity map, then
 the above notion coincides with classical differentiability.
 At a point $x$ where $f$ is differentiable, its differential is
 \[
 df_x(v)=\nabla f(x)\cdot v,
 \qquad v\in\mathbb{R}^n.
 \]
 
 Thus Rademacher's theorem implies that
 $(\mathbb{R}^n,\operatorname{id})$ is a Lipschitz differentiability
 chart for $\mathbb{R}^n$ equipped with Lebesgue measure.
 The identity chart has dimension $n$ and covers all of $\mathbb{R}^n$.
 Consequently, the pointwise analytic dimension satisfies
 \[
 d_C(x)=n
 \qquad\text{for $\mathcal{L}_n$-almost every $x\in\mathbb{R}^n$,}
 \]
 and its essential supremum, the analytic dimension, is $n$.

 Recall that the Heisenberg group $\mathbb H$ is $\mathbb R^3$ equipped with
 the group operation
 \[
 (x_1,y_1,z_1)\odot(x_2,y_2,z_2)
 =
 \left(x_1+x_2,\ y_1+y_2,\
 z_1+z_2+\frac12(x_1y_2-x_2y_1)\right).
 \]
 Its Carnot--Carath\'eodory distance $d_{CC}$ is left-invariant. The distance
 from the identity $(0,0,0)$ to $(p,q,r)$ is the infimum of
 \[
 \int_0^T\sqrt{\dot x(t)^2+\dot y(t)^2}\,dt
 \]
 over all absolutely continuous curves $(x(t),y(t))$ in $\mathbb R^2$
 that start at $(0,0)$, end at $(p,q)$, and satisfy
 \[
 \frac12\int_0^T
 \bigl(x(t)\dot y(t)-y(t)\dot x(t)\bigr)\,dt=r.
 \]
  Although
 $(\mathbb H,d_{CC})$ is homeomorphic to $\mathbb R^3$ and has Hausdorff
 dimension $4$, it has analytic dimension $2$. Indeed, the map
 \[
 \phi:\mathbb H\longrightarrow\mathbb R^2,
 \qquad \phi(x,y,z)=(x,y),
 \]
 is a global differentiability chart: for every Lipschitz function
 $f:\mathbb H\to\mathbb R$, at almost every $g\in\mathbb H$ there is a
 unique linear map $df_g:\mathbb R^2\to\mathbb R$ such that
 \[
 f(h)-f(g)
 =
 df_g\bigl(\phi(h)-\phi(g)\bigr)+o\bigl(d_{CC}(g,h)\bigr)
 \quad\text{as }h\to g.
 \]
 This follows from a version of Rademacher's theorem due to Pansu \cite{Pan}.
 Thus the pointwise dimension of the chart is $2$ almost everywhere,
 and the analytic dimension of $\mathbb H$ is $2$.

\subsection{PI spaces and Cheeger's theorem}

A large class of Lipschitz differentiability spaces arises from
Poincar\'e inequalities. Let $(X,d,\mu)$ be a metric measure space
such that every ball has positive, finite measure. We say that it
supports a $(1,p)$-Poincar\'e inequality if there are constants
$C>0$ and $\Lambda\geq 1$ such that, for every Lipschitz function
$f:X\to\mathbb R$ and every ball $B(x,r)$,
\[
\fint_{B(x,r)} |f-f_{B(x,r)}|\,d\mu
\leq
Cr\left(\fint_{B(x,\Lambda r)}
\operatorname{Lip}(f)^p\,d\mu\right)^{1/p},
\]
where
\[
\operatorname{Lip}(f)(y)
=
\limsup_{\substack{z\to y\\z\ne y}}
\frac{|f(z)-f(y)|}{d(z,y)},
\qquad
f_{B(x,r)}
=
\fint_{B(x,r)}f\,d\mu.
\]
If, in addition, $\mu$ is doubling, meaning that
\[
\mu(B(x,2r))\leq C_D\mu(B(x,r))
\quad\text{for all }x\in X,\ r>0,
\]
we call $(X,d,\mu)$ a $p$-PI space.

Cheeger's theorem \cite{Ch} states that a
complete $p$-PI space, with $1 \le p<\infty$, is a Lipschitz
differentiability space. Moreover, the dimensions of its charts
are uniformly bounded in terms of the doubling and Poincar\'e
inequality data. Thus Poincar\'e inequalities provide a way to
establish differentiability of Lipschitz functions without first
constructing coordinates explicitly.

The resulting class of spaces is quite broad. Laakso constructed,
for every $Q>1$, a $1$-PI space of Hausdorff dimension $Q$ \cite{Laa}.
The Heisenberg group $\mathbb H$, equipped with its
Carnot--Carath\'eodory metric and three-dimensional Lebesgue
measure, is another PI space; its Poincar\'e inequality follows
from Jerison's work \cite{Jer}.

Our understanding of PI spaces remains incomplete even when the
underlying topological space is $\bR^n$. By analogy with the study
of smooth structures on a fixed topological manifold, one may ask
which measurable differentiable structures can arise on $\bR^n$
when the metric and measure are allowed to vary. Here the metric
is required to induce the Euclidean topology, and the metric
measure space is required to satisfy the PI condition.
A first step toward such a classification is to determine the
possible analytic dimensions. Unlike a smooth structure, whose
dimension is fixed by the underlying topology, a measurable
differentiable structure need not have chart dimensions equal
to the topological dimension, as the Heisenberg group illustrates.
The following  two questions are due to Kleiner
and Schioppa \cite{KS}.

\begin{question}
	For each $n\geq2$, does there exist a complete PI space
	$(\bR^n,d,\mu)$ such that $d$ induces the Euclidean topology
	and the analytic dimension is one?   
\end{question}
Even the cases $n=2$ and $n=3$ are  open.
One can also ask which Hausdorff dimensions are compatible
with the PI condition and the prescribed topology.

\begin{question} 
	For each $n\geq2$ and $Q>n$, does there exist a complete
	PI space $(\bR^n,d,\mu)$ such that $d$ induces the Euclidean
	topology and $\mu$ is Ahlfors $Q$-regular? That is, do there
	exist such a space and a constant $C\geq1$ satisfying
	\[
	C^{-1}r^Q\leq\mu(B_d(x,r))\leq Cr^Q
	\qquad\text{for all $x\in\bR^n$ and $r>0$?}
	\]
\end{question}

The case $Q=n$ is provided by Euclidean space with Lebesgue
measure. When $n=3$, the case $Q=4$ is provided by the Heisenberg group. These questions ask how far the analytic and Hausdorff
dimensions can vary while the underlying topology remains fixed. There is a corresponding question for martingale dimension,
in which Gaussian heat kernel estimates provide a link between
the diffusion and the metric geometry as we will see below.

\section{Lecture 2: Equality under Gaussian heat kernel estimates}

Before comparing martingale and analytic dimensions, we highlight
a basic distinction between their definitions. Analytic dimension
depends on the metric through the notion of Lipschitz
differentiability, whereas martingale dimension is defined in terms
of the Dirichlet form and its energy measures, without reference
to a particular compatible metric. To relate these dimensions,
it is therefore natural to impose conditions that connect the
Dirichlet form to the metric. Heat kernel estimates provide such
a connection by describing the transition densities of the
associated diffusion in terms of distance and volume.
We refer the reader to \cite{Sal10} for an introduction to this
widely studied subject.

In this lecture, we compare martingale and analytic dimensions,
highlighting their similarities and differences, and discuss
their relationship with other notions of dimension, including
Hausdorff and spectral dimensions. Throughout Lecture 2, we assume that
$(X,d)$ is complete, separable, and has at least two points,
that the reference measure $m$ is Radon with full support,
and that the Dirichlet form $(\mathcal{E},\mathcal{F})$ on $L^2(X,m)$ is strongly local and regular. Furthermore, we assume that $m$ is a doubling measure on $(X,d)$.

 \subsection{Heat kernels and Gaussian estimates}
 
 Let $\{P_t\}_{t>0}$ denote the strongly continuous, symmetric
 Markov semigroup associated with $(\mathcal{E},\mathcal{F})$ on $L^2(X,m)$.
 The Dirichlet form can be recovered from the semigroup through
 \[
 \mathcal{F}=
 \left\{
 f\in L^2(X,m):
 \lim_{t\downarrow0}
 \frac{1}{t}\langle (I-P_t)f,f\rangle_{L^2(X,m)}<\infty
 \right\},
 \]
 and
 \[
 \mathcal{E}(f,f)=
 \lim_{t\downarrow0}
 \frac{1}{t}\langle (I-P_t)f,f\rangle_{L^2(X,m)},
 \qquad f\in\mathcal{F}.
 \]
 Thus, the Dirichlet form describes the infinitesimal behaviour
 of the semigroup.
 
 A \emph{heat kernel} for $\{P_t\}_{t>0}$ is a non-negative
 measurable function
 \[
 (t,x,y)\longmapsto p_t(x,y)
 \qquad\text{on $(0,\infty)\times X\times X$}
 \]
 such that, for every $t>0$ and $f\in L^2(X,m)$,
 \[
 P_tf(x)=\int_X p_t(x,y)f(y)\,m(dy)
 \qquad\text{for $m$-almost every $x\in X$.}
 \]
 It satisfies the symmetry and semigroup identities
 \[
 p_t(x,y)=p_t(y,x)
 \]
 and
 \[
 p_{t+s}(x,z)=
 \int_X p_t(x,y)p_s(y,z)\,m(dy),
 \]
 for each $t,s>0$ and almost every pair of spatial variables
 with respect to $m\otimes m$.
 Probabilistically, $p_t(x,y)$ is the transition density of
 the associated diffusion with respect to $m$.
 The existence of a heat kernel is an additional property;
 it does not follow from the Dirichlet form assumptions alone.
 
 \begin{definition}[Gaussian heat kernel estimates]
 	We say that $(\mathcal{E},\mathcal{F})$ satisfies \emph{two-sided Gaussian
 		heat kernel estimates} with respect to $(d,m)$ if it admits
 	a heat kernel and there exist constants
 	$c_1,c_2,C_1,C_2>0$ such that, for every $t>0$,
 	\begin{align*}
 		p_t(x,y)
 		&\geq
 		\frac{c_1}{m(B(x,\sqrt{t}))}
 		\exp\left(-C_1\frac{d(x,y)^2}{t}\right),\\
 		p_t(x,y)
 		&\leq
 		\frac{C_2}{m(B(x,\sqrt{t}))}
 		\exp\left(-c_2\frac{d(x,y)^2}{t}\right)
 	\end{align*}
 	for $m\otimes m$-almost every $(x,y)\in X\times X$.
 \end{definition}
 
 These estimates connect the diffusion to the metric and measure.
 The factor $m(B(x,\sqrt{t}))^{-1}$ gives the volume scale
 of the transition density, while the exponential factor describes
 its decay when $d(x,y)$ is large compared with $\sqrt{t}$.
 The term ``Gaussian'' refers to this dependence on
 $d(x,y)^2/t$; the volume factor need not be a power of $t$
 corresponding to an integer dimension.
 
 \begin{example}[Euclidean space]
 	For the Dirichlet form
 	\[
 	\mathcal{E}(f,g)=\int_{\bR^n}\nabla f\cdot\nabla g\,dx,
 	\qquad
 	\mathcal{F}=W^{1,2}(\bR^n),
 	\]
 	the heat kernel is
 	\[
 	p_t(x,y)=
 	\frac{1}{(4\pi t)^{n/2}}
 	\exp\left(-\frac{|x-y|^2}{4t}\right).
 	\]
 	Since $\mathcal{L}_n(B(x,\sqrt{t}))=\omega_n t^{n/2}$,
 	where $\omega_n$ is the volume of the Euclidean unit ball,
 	this kernel satisfies Gaussian estimates.
 \end{example}
 
 \begin{example}[The Heisenberg group]
 	The diffusion on $\bH$ introduced in Lecture~1, with generator
 	$\frac12(V_1^2+V_2^2)$, satisfies Gaussian estimates with
 	respect to the Carnot--Carath\'eodory distance $d_{CC}$ and
 	three-dimensional Lebesgue measure $\mathcal{L}_3$.
 	The Heisenberg dilations imply that
 	\[
 	\mathcal{L}_3(B_{CC}(x,r))=c r^4
 	\]
 	for a constant $c>0$. Consequently, the volume factor in
 	the Gaussian estimates is proportional to $t^{-2}$.
 	Thus, although the underlying space has topological dimension
 	three and the diffusion has martingale dimension two, its
 	heat kernel reflects volume growth of dimension four.
 \end{example}
 
 Gaussian estimates hold in many other settings; see \cite{Sal10}
 for background and references. Geometric examples include
 Brownian motion on connected Lie groups of polynomial volume
 growth, equipped with a left-invariant Riemannian metric and
 its Riemannian volume measure, as well as Brownian motion on
 complete connected Riemannian manifolds without boundary
 and with non-negative Ricci curvature, by the work of Li and Yau. Another classical class of examples comes from symmetric
 uniformly elliptic divergence-form operators on $\bR^n$.
 More precisely, the form
 \[
 \mathcal{E}(f,g)=\int_{\bR^n}
 \langle A(x)\nabla f(x),\nabla g(x)\rangle\,dx,
 \qquad \mathcal{F}=W^{1,2}(\bR^n),
 \]
 satisfies Gaussian estimates if $A$ is measurable and symmetric
 and, for some $0<\lambda\leq\Lambda<\infty$,
 \[
 \lambda|\xi|^2
 \leq \langle A(x)\xi,\xi\rangle
 \leq \Lambda|\xi|^2
 \]
 for almost every $x$ and every $\xi\in\bR^n$ \cite{Aro}.
 
 More generally, consider a complete connected weighted
 Riemannian manifold without boundary, with a smooth positive
 weight, equipped with its Riemannian distance and weighted
 volume measure $m$. If volume doubling and a scale-invariant
 $L^2$-Poincar\'e inequality hold, then the form
 \[
 \mathcal{E}(f,g)=\int_M\langle\nabla f,\nabla g\rangle\,dm,
 \]
 with its natural Sobolev domain, satisfies Gaussian estimates.
 This is a celebrated result of Grigor'yan and Saloff-Coste
 \cite{Gri,Sal92}; see also \cite{Stu} for an extension to
 strongly local Dirichlet spaces. Examples beyond the manifold
 setting include Laakso spaces equipped with their natural
 metrics, measures, and canonical Dirichlet forms.
 
 \subsection{Equality of the pointwise dimensions}
 Under Gaussian heat kernel bounds the pointwise versions of martingale and analytic dimensions agree \cite{murugananalytic}. In this case, a result of Koskela and Zhou implies that the reference measure $m$ is a minimal energy dominant measure \cite{KZ} and that the underlying metric measure space is a PI space. The following result is from \cite{murugananalytic}.
 
 \begin{theorem}\label{thm:equality}  
 	Let $(\mathcal{E},\mathcal{F})$ be a Dirichlet form on $L^2(X,m)$  satisfying two-sided Gaussian heat kernel estimates with respect to the metric $d$.  Let $p_H$ be the
 	pointwise index of the form and $d_C$ the pointwise dimension of Cheeger's
 	measurable differentiable structure of the underlying metric measure space $(X,d,m)$.  Then
 	\[
 	p_H(x)=d_C(x)
 	\qquad\text{for $m$-almost every }x\in X.
 	\]
 	In particular, the martingale dimension of the diffusion equals the analytic
 	dimension of $(X,d,m)$.
 \end{theorem}

 The key bridge between the different dimensions is the following estimate, essentially due to Koskela and Zhou
 \cite{KZ} together with the comparison of the underlying metric with the intrinsic metric established in \cite{KM20}.
 
 \begin{proposition}[Energy--slope comparison]\label{prop:energy-slope}
 	Assume two-sided Gaussian heat kernel estimates.  Then every Lipschitz
 	function belongs locally to the form domain, the measure $m$ is minimal energy dominant, and there is $C\geq1$ such that
 	\[
 	C^{-1}\left(\frac{d\Gamma(f,f)}{d m}(x)\right)^{1/2}
 	\leq \Lip f(x)\leq
 	C\left(\frac{d\Gamma(f,f)}{d m}(x)\right)^{1/2}
 	\]
 	for every Lipschitz function $f:X \to \mathbb{R}$ and for $m$-almost every $x \in X$.
 \end{proposition}
 
 This proposition says that the metric and energy notions of
 the size of a first-order derivative are quantitatively equivalent.
 In particular, for Lipschitz functions $f_1,\ldots,f_N$, the matrix
 \[
 \left(\frac{d\Gamma(f_i,f_j)}{dm}(x)\right)_{i,j=1}^N
 \]
 has full rank if and only if
 $\Lip(\sum_{i=1}^N a_i f_i)(x)>0$ for every non-zero
 $a\in\bR^N$, at $m$-almost every $x$.
 Thus, independence can be expressed either through the rank
 of the energy matrix or through the pointwise slopes of
 linear combinations. This observation is behind the proof of Theorem \ref{thm:equality}.
 
 \subsection{Estimates on martingale dimension}
Cheeger conjectured that the analytic dimension of a PI space
is bounded above by its Hausdorff dimension \cite{Ch}.
De Philippis, Marchese, and Rindler \cite{DMR} proved the
stronger statement that, for every $n$-dimensional chart
$(U,\phi)$ of a Lipschitz differentiability space $(X,d,m)$,
\[
\phi_*(\one_U\,m)\ll\mathcal{L}_n.
\]
Since $m(U)>0$, this implies that $\mathcal{L}_n(\phi(U))>0$,
and hence $\dim_H\phi(U)=n$. As Lipschitz maps do not increase
Hausdorff dimension,
\[
n=\dim_H\phi(U)\leq\dim_H U\leq\dim_H X.
\]
Taking the supremum over the chart dimensions gives the
desired bound on analytic dimension. Combining this bound with Theorem~\ref{thm:equality}, we obtain
that, under two-sided Gaussian heat kernel estimates, the martingale
dimension of a strongly local regular Dirichlet space is bounded
above by the Hausdorff dimension of the underlying metric space.

 \subsection{Sub-Gaussian heat kernel estimates}
 
 Gaussian estimates are a special case of a broader class of
 heat kernel estimates that arise naturally in the study of
 diffusions on fractals \cite{Bar}. 
 
 \begin{definition}[Sub-Gaussian heat kernel estimates]
 	Let $\beta\in[2,\infty)$. We say that $(\mathcal{E},\mathcal{F})$ satisfies
 	\emph{two-sided sub-Gaussian heat kernel estimates with walk
 		dimension $\beta$} if it admits a heat kernel and there exist
 	constants $c_1,c_2,C_1,C_2>0$ such that, for every $t>0$,
 	\begin{align*}
 		p_t(x,y)
 		&\geq
 		\frac{c_1}{m(B(x,t^{1/\beta}))}
 		\exp\left(
 		-C_1\left(\frac{d(x,y)^\beta}{t}\right)^{1/(\beta-1)}
 		\right),\\
 		p_t(x,y)
 		&\leq
 		\frac{C_2}{m(B(x,t^{1/\beta}))}
 		\exp\left(
 		-c_2\left(\frac{d(x,y)^\beta}{t}\right)^{1/(\beta-1)}
 		\right)
 	\end{align*}
 	for $m\otimes m$-almost every $(x,y)\in X\times X$.
 \end{definition}
 
 When $\beta=2$, these are precisely the Gaussian estimates.
 For general $\beta$, the characteristic spatial scale at
 time $t$ is $t^{1/\beta}$, corresponding to a time scale
 $r^\beta$ for motion over distance $r$.
 
 Suppose additionally that $m$ is \emph{Ahlfors $\alpha$-regular},
 meaning that
 \[
 m(B(x,r))\asymp r^\alpha
 \qquad
 \text{for all $x\in X$ and $0<r<\diam(X,d)$,}
 \]
 with comparison constants independent of $x$ and $r$.
 We call $\alpha$ the \emph{volume growth exponent}; in this
 case, $\alpha=\dim_H(X,d)$.
 The \emph{spectral dimension} is
 \[
 d_s:=\frac{2\alpha}{\beta}.
 \]
 This terminology reflects the on-diagonal behaviour of the
 heat kernel: using its continuous version, the estimates give
 \[
 p_t(x,x)\asymp t^{-\alpha/\beta}=t^{-d_s/2},
 \qquad 0<t<\diam(X,d)^\beta.
 \]
 Thus, under Gaussian estimates, $d_s=\dim_H(X,d)$, whereas
 for $\beta>2$ the spectral dimension is strictly smaller
 than the Hausdorff dimension.
 
 A classical example is Brownian motion on the standard
 Sierpi\'nski gasket. Barlow and Perkins  \cite{BP} established
 sub-Gaussian estimates with
 \[
 \alpha=\frac{\log3}{\log2},
 \qquad
 \beta=\frac{\log5}{\log2},
 \qquad
 d_s=\frac{2\log3}{\log5}\in(1,2).
 \]
 The work of Barlow and Bass also established sub-Gaussian
 estimates for Brownian motion on Sierpi\'nski carpet and higher dimensional generalizations called generalized Sierpi\'nski carpets \cite{BB99}.
 For generalized Sierpi\'nski carpets, the walk dimension
 is strictly greater than two, although its exact value
 is generally unknown \cite{KajinoWalk}.  
 
 We now turn to martingale dimension in this broader setting.
 Recall that the martingale dimension $d_m$ is the essential
 supremum of Hino's pointwise index with respect to a minimal
 energy dominant measure. If $\mathcal{E}\not\equiv0$, then
 \[
 d_m\geq1.
 \]
 Obtaining stronger lower bounds for specific fractal diffusions
 is much more difficult. In particular, the following question
 remains open.
 
 \begin{question}
 	Is there a generalized Sierpi\'nski carpet whose canonical
 	self-similar Dirichlet form has martingale dimension strictly
 	greater than one?
 \end{question}
 
 Hino \cite{Hin13a} proved the bound $d_m\leq d_s$ for a class
 of self-similar spaces and Dirichlet forms. We expect this bound
 to hold without any self-similarity assumption; see also Hino's
 recent work \cite{Hin25} for progress toward determining
 martingale dimension beyond the self-similar setting. The
 following conjecture may be viewed as a counterpart of the
 bound on analytic dimension by Hausdorff dimension for PI spaces \cite{DMR}.
 
 \begin{conjecture}\label{c:ub-mart}
 	Let $(X,d,m)$ be an Ahlfors $\alpha$-regular metric measure
 	space, and let $(\mathcal{E},\mathcal{F})$ be a strongly local, regular
 	Dirichlet form on $L^2(X,m)$ satisfying two-sided sub-Gaussian
 	heat kernel estimates with walk dimension $\beta\geq2$.
 	Then
 	\[
 	d_m\leq d_s=\frac{2\alpha}{\beta}.
 	\]
 \end{conjecture}
 
 When $\beta=2$, this bound follows from
 Theorem~\ref{thm:equality} and the bound on analytic dimension
 by Hausdorff dimension. The conjecture asks whether spectral
 dimension provides the corresponding upper bound in the
 sub-Gaussian setting. The following recent result of the
 authors \cite{ErikssonBiqueMuruganEID} gives the weaker
 qualitative conclusion that the martingale dimension is finite.
 
 \begin{theorem} \label{t:finite-mtg-dim}
 	Let $(X,d,m)$ be a metric measure space with $m$ doubling,
 	and let $(\mathcal{E},\mathcal{F})$ be a strongly local, regular Dirichlet
 	form on $L^2(X,m)$ satisfying two-sided sub-Gaussian heat
 	kernel estimates with walk dimension $\beta$. Then the
 	martingale dimension of $(\mathcal{E},\mathcal{F})$ is finite.
 \end{theorem}
 
 Hino's bound and known constructions further suggest the
 following conjectural characterization of the possible triples
 $(\alpha,\beta,d_m)$.
 
 \begin{conjecture}
 	Let $\alpha\geq1$, $\beta\geq2$, and $d_m\in\bN$. There exists
 	an Ahlfors $\alpha$-regular metric measure space carrying a
 	strongly local, regular Dirichlet form with two-sided
 	sub-Gaussian heat kernel estimates of walk dimension $\beta$
 	and martingale dimension $d_m$ if and only if
 	\[
 	2\leq\beta\leq1+\alpha
 	\qquad\text{and}\qquad
 	1\leq d_m\leq\frac{2\alpha}{\beta}.
 	\]
 \end{conjecture}
 
 For $d_m=1$, the existence assertion is established by the
 Laakso-type constructions in \cite{Mur-inv}, which realize
 every pair $\alpha\geq1$ and $2\leq\beta\leq1+\alpha$ with
 martingale dimension one.
 
 There are also examples of spaces satisfying sub-Gaussian heat
 kernel estimates for which the form domain contains no
 non-constant Lipschitz functions. This indicates that Lipschitz
 functions need not capture the first-order structure of a
 general strongly local Dirichlet form. In particular, analytic
 dimension is not generally suitable for describing the
 martingale dimension of fractal diffusions with $\beta>2$.
 
 For example, the standard Sierpi\'nski gasket and Sierpi\'nski
 carpet, equipped with their Euclidean metrics and Hausdorff
 measures, are not Lipschitz differentiability spaces. Indeed,
 they are Ahlfors regular of non-integer Hausdorff dimension and
 are bi-Lipschitz embedded in Euclidean space, whereas a theorem
 of Kell and Mondino \cite{Kell-Mon} asserts that a Lipschitz
 differentiability space that admits such an embedding must be
 rectifiable.
 
 More generally, Theorem~\ref{t:diff-space} shows that if a
 metric measure space carries a strongly local, regular
 Dirichlet form satisfying two-sided sub-Gaussian heat kernel
 estimates with walk dimension $\beta>2$, then it is not a
 Lipschitz differentiability space. This is in sharp contrast
 with the Gaussian case $\beta=2$ considered in
 Theorem~\ref{thm:equality}.

\section{Lecture 3: The energy image density property}

A common mechanism behind several of the dimension bounds
discussed in Lecture~2 is an image-density principle:
the presence of sufficiently many independent first-order
directions forces an associated image measure on $\bR^n$
to be absolutely continuous with respect to Lebesgue measure.
In the setting of Dirichlet forms, this principle takes the
form of the energy image density (EID) property. Versions of
this idea play an important role in the bounds for analytic
dimension in \cite{DMR} and for martingale dimension in
\cite{Hin13a,ErikssonBiqueMuruganEID}.

In this lecture, we trace the origins of the EID property to
Malliavin's proof of H\"ormander's hypoellipticity theorem \cite{Hor,Mal}.
We then introduce the EID conjecture formulated by Bouleau
and Hirsch \cite{bouleauhirsch}, explain its recent resolution
in \cite{ErikssonBiqueMuruganEID}, and discuss how the EID
property yields upper bounds on martingale and analytic
dimensions.

\subsection{H\"ormander's theorem and Malliavin's criterion}

A linear differential operator $L$ with smooth coefficients is
called \emph{hypoelliptic} if, for every open set
$U\subset\bR^n$ and every distribution $u\in\mathcal{D}'(U)$,
\[
Lu\in C^\infty(U)
\qquad\Longrightarrow\qquad
u\in C^\infty(U).
\]
Elliptic operators, such as the Laplacian, are hypoelliptic;
Weyl's lemma is a classical manifestation of this fact. The heat
operator $\partial_t-\Delta$ is also hypoelliptic, whereas the
wave operator $\partial_{tt}-\Delta$ is not.

An important non-elliptic example is the Kolmogorov operator
\[
K=\partial_{vv}+v\partial_x-\partial_t,
\]
introduced by Kolmogorov in 1934. Here $t$, $x$, and $v$
represent time, position, and velocity, respectively. This
operator is associated with the diffusion
\[
dX_s=V_s\,ds,
\qquad
dV_s=\sqrt{2}\,dB_s,
\]
where $B$ is a one-dimensional Brownian motion. Randomness acts
directly only in the velocity variable, but is transmitted to
the position variable through the equation $dX_s=V_s\,ds$.
Nevertheless, $K$ is hypoelliptic. H\"ormander's theorem explains this phenomenon. Consider an
operator
\[
L=\sum_{j=1}^d V_j^2+V_0,
\]
where $V_0,V_1,\ldots,V_d$ are smooth real vector fields on an
open set $U\subset\bR^n$. Recall that the Lie bracket of two
vector fields is defined by
\[
[V,W]=VW-WV.
\]
H\"ormander's bracket condition requires that the vector fields
$V_0,V_1,\ldots,V_d$, together with all of their iterated Lie
brackets, span the tangent space at every point of $U$.

\begin{theorem}[H\"ormander \cite{Hor}]
	If H\"ormander's bracket condition holds, then
	\[
	L=\sum_{j=1}^d V_j^2+V_0
	\]
	is hypoelliptic.
\end{theorem}

For the Kolmogorov operator, take
\[
V_0=v\partial_x-\partial_t,
\qquad
V_1=\partial_v.
\]
Since
\[
[V_0,V_1]=-\partial_x,
\]
the vector fields $V_0$, $V_1$, and $[V_0,V_1]$ span the
$t$, $v$, and $x$ directions. Thus H\"ormander's theorem
implies that the Kolmogorov operator is hypoelliptic.

There is a corresponding probabilistic formulation. Consider
the Stratonovich stochastic differential equation
\[
dX_t
=
\sqrt{2}\sum_{i=1}^d V_i(X_t)\circ dB_t^i
+V_0(X_t)\,dt,
\]
where $B^1,\ldots,B^d$ are independent Brownian motions.
The generator of this diffusion is
\[
L=\sum_{i=1}^d V_i^2+V_0.
\]
The relevant parabolic H\"ormander condition is that the Lie
algebra generated on $\bR\times\bR^n$ by
\[
\partial_t-V_0,\ V_1,\ldots,V_d
\]
spans the tangent space at every point. This condition implies that, for
every $t>0$, the law of $X_t$ has a smooth density with respect
to Lebesgue measure.

The Brownian motion on the Heisenberg group, discussed in
Lecture~1, provides a useful example. Let
$(B_t^{(1)},B_t^{(2)})$ be a two-dimensional Brownian motion
and set
\[
X_t=
\left(
B_t^{(1)},B_t^{(2)},
\frac12\int_0^t
\left(
B_s^{(1)}\,dB_s^{(2)}
-
B_s^{(2)}\,dB_s^{(1)}
\right)
\right).
\]
As explained in Lecture 1, the generator of $X_t$ is
\[
L=\frac12(V_1^2+V_2^2),
\qquad
V_1=\partial_x-\frac{y}{2}\partial_z,
\qquad
V_2=\partial_y+\frac{x}{2}\partial_z.
\]
Although the process is driven by only two independent Brownian
motions, the bracket
\[
[V_1,V_2]=\partial_z
\]
produces the third spatial direction. Hence H\"ormander's
condition holds, and $X_t$ has a smooth density on $\bR^3$
for every $t>0$.

Malliavin's probabilistic proof of H\"ormander's theorem begins
by viewing the solution at time $t$ as a functional of the
driving Brownian path:
\[
X_t=F_t(B).
\]
Thus,
\[
F_t:C_0([0,t],\bR^d)\longrightarrow\bR^n
\]
is a map defined on Wiener space. Malliavin calculus studies
how $F_t(B)$ changes when the Brownian path is perturbed along
directions in the Cameron--Martin space $H$. Suppose
\[
F=(F_1,\ldots,F_n)
\]
has components in a suitable Malliavin Sobolev space. Its Malliavin matrix is
\[
\gamma_F
=
\left(
\langle DF_i,DF_j\rangle_H
\right)_{i,j=1}^n,
\]
where $DF_i$ denotes the Malliavin derivative of $F_i$.
The matrix $\gamma_F$ is invertible precisely when
\[
DF:H\longrightarrow\bR^n
\]
is surjective. The one-derivative Malliavin criterion states that
\[
F_*\left(
\one_{\{\det\gamma_F>0\}}\,\mathbb{P}
\right)
\ll\mathcal{L}_n.
\]
In particular, if $\det\gamma_F>0$ almost surely, then the law
$F_*\mathbb{P}$ is absolutely continuous with respect to
Lebesgue measure. In Malliavin's proof of H\"ormander's theorem,
the bracket condition is used to establish the non-degeneracy
of the Malliavin matrix. Stronger estimates on its inverse,
combined with Malliavin's integration-by-parts formula, then
give smoothness of the density.

This criterion is the prototype for the energy image density
property. On Wiener space, the Malliavin matrix is the carré
du champ matrix of the Ornstein--Uhlenbeck Dirichlet form.
The EID property replaces the Wiener-space Malliavin matrix by
the carré du champ matrix of a general Dirichlet structure
while retaining the same image-measure conclusion.

 \subsection{EID conjecture of Bouleau and Hirsch}
The finite-dimensional analogue of Malliavin's criterion follows
from the coarea formula. Let
\[
f:\bR^m\longrightarrow\bR^n,
\qquad m\geq n,
\]
be a smooth map, and define its $n$-dimensional Jacobian by
\[
J_f(x)
:=
\sqrt{\det\bigl(Df(x)Df(x)^{\mathsf T}\bigr)}.
\]
The matrix $Df(x)Df(x)^{\mathsf T}$ is the Gram matrix of the
component gradients of $f$. Thus,
\[
J_f(x)>0
\quad\Longleftrightarrow\quad
Df(x):\bR^m\longrightarrow\bR^n
\text{ is surjective}.
\]

The coarea formula states that, for every non-negative measurable
function $g:\bR^m\to[0,\infty]$,
\[
\int_{\bR^m}g(x)J_f(x)\,dx
=
\int_{\bR^n}
\left(
\int_{f^{-1}(y)}
g(x)\,d\mathcal{H}_{m-n}(x)
\right)dy.
\]

Now let $X$ be an $\bR^m$-valued random variable with density
$\rho$ with respect to $\mathcal{L}_m$. Applying the coarea
formula with $g=\rho/J_f$ on the set $\{J_f>0\}$ shows that
\[
f_*\bigl(\one_{\{J_f>0\}}\rho\,\mathcal{L}_m\bigr)
\ll\mathcal{L}_n.
\]
Its density is given by
\[
\sigma(y)
=
\int_{f^{-1}(y)\cap\{J_f>0\}}
\frac{\rho(x)}{J_f(x)}
\,d\mathcal{H}_{m-n}(x).
\]
In particular, if $f$ is a submersion, so that $J_f>0$
everywhere, then the law of $f(X)$ is absolutely continuous
and
\[
f_*(\rho\,\mathcal{L}_m)
=
\sigma\,\mathcal{L}_n.
\]

This is the finite-dimensional prototype of Malliavin's
criterion. The matrix $DfDf^{\mathsf T}$ is replaced on Wiener
space by the Malliavin matrix, and its non-degeneracy again
forces the corresponding image measure to be absolutely
continuous. The EID property extends this principle further
by replacing the Malliavin matrix with the carré du champ
matrix of an abstract Dirichlet structure.

The preceding finite-dimensional discussion suggests an
abstraction in which the carré du champ takes the place of the
Gram matrix $DfDf^{\mathsf T}$.

\begin{definition}[Dirichlet structure]
	A \emph{Dirichlet structure}
	$(X,\mathcal{X},\mu,\mathcal{E},\mathcal{F})$ consists of a probability space
	$(X,\mathcal{X},\mu)$ and a symmetric Dirichlet form
	$(\mathcal{E},\mathcal{F})$ on $L^2(X,\mu)$ satisfying the following additional
	properties:
	\begin{enumerate}[(i)]
		\item The form is strongly local: if $f,g\in\mathcal{F}$ and
		$a\in\bR$ satisfy $(f+a)g=0$, then $\mathcal{E}(f,g)=0$.
		
		\item The constant function $\one$ belongs to $\mathcal{F}$ and
		$\mathcal{E}(\one,\one)=0$.
		
		\item The form admits a \emph{carr\'e du champ}; that is, there
		exists a positive symmetric bilinear map
		\[
		\gamma:\mathcal{F}\times\mathcal{F}\longrightarrow L^1(\mu)
		\]
		such that, for every $f,h\in\mathcal{F}\cap L^\infty(\mu)$,
		\[
		\mathcal{E}(fh,f)-\frac12\mathcal{E}(h,f^2)
		=
		\int_X h\,\gamma(f,f)\,d\mu.
		\]
	\end{enumerate}
\end{definition}

For an $\bR^n$-valued function
$f=(f_1,\ldots,f_n)\in\mathcal{F}^n$, its \emph{carr\'e du champ
	matrix} is
\[
\gamma(f)
:=
\bigl(\gamma(f_i,f_j)\bigr)_{1\leq i,j\leq n}.
\]
This is a measurable, symmetric, positive semidefinite
$n\times n$ matrix. The condition
\[
\det\gamma(f)(x)>0
\]
therefore expresses the infinitesimal independence of the
components $f_1,\ldots,f_n$ at $x$.

\begin{definition}[Energy image density property]
	A Dirichlet structure satisfies the \emph{energy image density
		property}, abbreviated EID, if, for every $n\in\bN$ and every
	$f\in\mathcal{F}^n$,
	\[
	f_*\bigl(\one_{\{\det\gamma(f)>0\}}\,\mu\bigr)
	\ll\sL_n.
	\]
\end{definition}

Thus, after restricting $\mu$ to the set on which the components
of $f$ are infinitesimally independent, its image under $f$ is
absolutely continuous with respect to Lebesgue measure.

As a finite-dimensional example, let $\mu$ be the standard
Gaussian measure on $\bR^m$ and consider the
Ornstein--Uhlenbeck Dirichlet structure
\[
\mathcal{F}=W^{1,2}(\bR^m,\mu),
\qquad
\mathcal{E}(u,v)=\int_{\bR^m}\nabla u\cdot\nabla v\,d\mu.
\]
Its carr\'e du champ is
\[
\gamma(u,v)=\nabla u\cdot\nabla v.
\]
Consequently, for a smooth map
$f=(f_1,\ldots,f_n):\bR^m\to\bR^n$,
\[
\gamma(f)=DfDf^{\mathsf T}.
\]
Hence $\det\gamma(f)(x)>0$ precisely when $Df(x)$ is
surjective. In this case, the EID property recovers the
absolute-continuity statement obtained from the coarea formula.
For the Ornstein--Uhlenbeck Dirichlet structure on Wiener space,
the matrix $\gamma(f)$ is the classical Malliavin matrix.
Motivated by Malliavin's criterion,
Bouleau and Hirsch formulated the following conjecture.

\begin{conjecture}[Bouleau--Hirsch, 1986
	\cite{bouleauhirsch}]
	Every Dirichlet structure satisfies the energy image density
	property.
\end{conjecture}

This conjecture was recently resolved affirmatively \cite{ErikssonBiqueMuruganEID}.
\begin{theorem}  \label{t:EID}
	Every Dirichlet structure satisfies the energy image density
	property.
\end{theorem}

The result also extends beyond the original framework of
Dirichlet structures. Let $(\mathcal{E},\mathcal{F})$ be a strongly local,
regular Dirichlet form, let $\nu$ be a minimal energy dominant
measure, and define
\[
\gamma_\nu(f)
=
\left(
\frac{d\Gamma(f_i,f_j)}{d\nu}
\right)_{1\leq i,j\leq n}.
\]
Then, using quasicontinuous representatives of the components
of $f$,
\[
f_*\bigl(
\one_{\{\det\gamma_\nu(f)>0\}}\,\nu
\bigr)
\ll\sL_n.
\]
Versions of the EID property are also established for
$p$-Dirichlet spaces with $1<p<\infty$. These include Sobolev
spaces defined using upper gradients and self-similar
$p$-energies on fractals obtained as rescaled limits of
discrete energies.

  Several important special cases of the EID conjecture were
  known previously. The scalar case $n=1$ was established by
  Bouleau and by Bouleau--Hirsch
  \cite{Bou,bouleauhirsch}. Bouleau and Hirsch also proved the
  EID property for the Ornstein--Uhlenbeck Dirichlet structure
  on Wiener space \cite{BH86b}. In that setting, the carr\'e du
  champ matrix is the Malliavin matrix, and the resulting
  absolute-continuity statement is commonly called the
  Bouleau--Hirsch criterion.
   
   An important feature of the Bouleau--Hirsch criterion is that it
   requires only first-order Malliavin regularity: if the components of
   $f$ have first-order Malliavin regularity and
   $\det\gamma(f)>0$ almost surely, then the law of $f$ is absolutely
   continuous. By contrast, Malliavin's original integration-by-parts
   approach required additional regularity.
  The two classical approaches to image-density results may be
  summarized as follows.
  
  Malliavin's approach uses an integration-by-parts formula to
  show that suitable distributional derivatives of the image
  measure are Radon measures. In the setting of general
  Dirichlet forms, this method typically requires additional
  regularity of $f$ and of the entries of its carr\'e du champ
  matrix. It therefore does not apply to arbitrary
  $f\in\mathcal{F}^n$.
  The approach of Bouleau and Hirsch instead uses Federer's
  coarea formula. Its advantage is that it requires only
  first-order regularity and consequently yields the
  one-derivative criterion on Wiener space. Its limitation is
  that a coarea formula is not available for an arbitrary
  Dirichlet structure. Extensions of this approach therefore
  require additional geometric or structural assumptions.
  
  A further partial result was obtained by Malicet and Poly
  \cite{MP}. They proved that if $f\in\mathcal{F}^n$ satisfies
  \[
  \det\gamma(f)>0
  \qquad\text{$\mu$-almost surely},
  \]
  then its law $f_*\mu$ is a Rajchman measure; that is, its
  Fourier transform tends to zero at infinity. This conclusion
  is weaker than absolute continuity, since a Rajchman measure
  may still be singular with respect to Lebesgue measure.
  
  Thus, the central difficulty in the EID conjecture was to show,
  using only the abstract first-order structure of a Dirichlet
  form, that non-degeneracy of the carr\'e du champ matrix forces
  absolute continuity of the corresponding image measure,
  without assuming the additional regularity needed for
  integration by parts or the geometric structure needed for
  a coarea formula.

Next we discuss some ideas behind the proof of Theorem \ref{t:EID}.
The proof of the EID conjecture draws on tools developed in
connection with a converse to Rademacher's theorem. Recall that
Rademacher's theorem implies that, if $\nu$ is a Radon measure
on $\bR^n$ satisfying
$\nu\ll\sL_n$,
then every Lipschitz function $g:\bR^n\to\bR$ is differentiable
at $\nu$-almost every point. De Philippis and Rindler
\cite{DR} proved the converse: if every Lipschitz function
$g:\bR^n\to\bR$ is differentiable $\nu$-almost everywhere,
then
\[
\nu\ll\sL_n.
\]
Our proof of the EID conjecture uses a criterion for absolute
continuity arising from the same paper.

We first recall briefly the language of currents. A
$k$-dimensional current $T$ in $\bR^n$ is a continuous linear
functional on the space $\mathcal{D}^k(\bR^n)$ of smooth,
compactly supported $k$-forms. Its boundary is the
$(k-1)$-current defined by
\[
\partial T(\omega):=T(d\omega),
\qquad
\omega\in\mathcal{D}^{k-1}(\bR^n),
\]
and its mass is
\[
\mathbf{M}(T)
:=
\sup\left\{
T(\omega):
\omega\in\mathcal{D}^k(\bR^n),\
|\omega|\leq1
\right\},
\]
where $|\omega|$ denotes the pointwise comass norm. For example,
if $T_\Omega$ is the current of integration over a smoothly
bounded, oriented domain $\Omega\subset\bR^n$, then
$\partial T_\Omega$ is the current of integration over
$\partial\Omega$, with the induced orientation, and
\[
\mathbf{M}(T_\Omega)=\sL_n(\Omega),
\qquad
\mathbf{M}(\partial T_\Omega)
=\sH_{n-1}(\partial\Omega).
\]

A current $T$ is called \emph{normal} if both $T$ and
$\partial T$ have finite mass. We shall only use
one-dimensional currents. Every $1$-current of finite mass
can be represented as
\[
T=\vec T\,\|\!T\!\|,
\]
where $\|\!T\!\|$ is a finite positive measure and
$\vec T:\bR^n\to\bR^n$ is measurable with
$|\vec T|=1$ for $\|\!T\!\|$-almost every point. More
explicitly,
\[
T(\omega)
=
\int_{\bR^n}
\langle\omega(y),\vec T(y)\rangle\,d\|\!T\!\|(y).
\]

The following consequence of the structure theorem of
De Philippis and Rindler is the main geometric ingredient.

\begin{theorem}[De Philippis--Rindler \cite{DR}]
	Let
	\[
	T_i=\vec T_i\,\|\!T_i\!\|,
	\qquad i=1,\ldots,n,
	\]
	be one-dimensional normal currents in $\bR^n$, and let $\nu$
	be a positive Radon measure such that
	\[
	\nu\ll\|\!T_i\!\|
	\qquad\text{for every $i=1,\ldots,n$.}
	\]
	If
	\[
	\operatorname{span}
	\{\vec T_1(y),\ldots,\vec T_n(y)\}
	=\bR^n
	\qquad\text{for $\nu$-almost every $y$,}
	\]
	then
	\[
	\nu\ll\sL_n.
	\]
\end{theorem}

We now explain how the Dirichlet structure produces the normal
currents required by this theorem. Let
$A:D(A)\to L^2(X,\mu)$ denote the generator of the Dirichlet
form, with our sign convention
\[
\mathcal{E}(g,h)=-\langle A(g),h\rangle_{L^2(\mu)},
\qquad g\in D(A),\quad h\in\mathcal{F}.
\]
Equivalently,
\[
D(A)
=
\left\{
g\in\mathcal{F}:
|\mathcal{E}(g,h)|\leq C\|h\|_{L^2(\mu)}
\text{ for some $C>0$ and every $h\in\mathcal{F}$}
\right\}.
\]
The energy measures of the Dirichlet structure are
\[
\Gamma(u,v):=\gamma(u,v)\,\mu.
\]

Given $f=(f_1,\ldots,f_n)\in\mathcal{F}^n$ and $g\in\mathcal{F}$, define a
one-dimensional current $T_{f,g}$ in $\bR^n$ by
\[
T_{f,g}\left(\sum_{j=1}^n h_j\,dx_j\right)
:=
\sum_{j=1}^n
\int_X h_j(f(x))\,d\Gamma(f_j,g)(x).
\]
Equivalently, $T_{f,g}$ is the $\bR^n$-valued measure
\[
T_{f,g}
=
f_*\bigl(
(\Gamma(f_j,g))_{j=1}^n
\bigr).
\]
The Cauchy--Schwarz inequality for energy measures shows that
this current has finite mass.

The crucial observation is that $T_{f,g}$ is normal whenever
$g\in D(A)$. Indeed, if
$\phi\in C_c^\infty(\bR^n)$ and
$\phi_0=\phi-\phi(0)$, then the chain rule gives
\begin{align*}
	\partial T_{f,g}(\phi)
	&=
	T_{f,g}(d\phi)\\
	&=
	\sum_{j=1}^n
	\int_X
	\frac{\partial\phi}{\partial x_j}(f(x))
	\,d\Gamma(f_j,g)(x)\\
	&=
	\mathcal{E}(\phi_0\circ f,g)\\
	&=
	-\int_X A(g)\,\phi_0(f)\,d\mu.
\end{align*}
Moreover,
\[
\int_X A(g)\,d\mu
=
-\mathcal{E}(g,\one)
=0.
\]
It follows that
\[
\partial T_{f,g}(\phi)
=
-\int_X A(g)\,\phi(f)\,d\mu,
\]
or equivalently,
\[
\partial T_{f,g}
=
-\,f_*\bigl(A(g)\,\mu\bigr).
\]
Thus $\partial T_{f,g}$ has finite mass and $T_{f,g}$ is a
normal current.

To apply the De Philippis--Rindler theorem, fix
$f=(f_1,\ldots,f_n)\in\mathcal{F}^n$ and set
\[
I_f:=\{x\in X:\det\gamma(f)(x)>0\},
\qquad
\nu_f:=f_*(\one_{I_f}\mu).
\]
For each $i=1,\ldots,n$, consider the current
\[
T_i:=T_{f,f_i}.
\]
The invertibility of the carr\'e du champ matrix implies
\[
\nu_f\ll\|\!T_i\!\|,
\qquad i=1,\ldots,n,
\]
and
\[
\operatorname{span}
\{\vec T_1(y),\ldots,\vec T_n(y)\}
=\bR^n
\qquad\text{for $\nu_f$-almost every $y$.}
\]
Thus, if $f\in D(A)^n$, the currents $T_1,\ldots,T_n$ are
normal and the De Philippis--Rindler theorem yields
\[
\nu_f
=
f_*(\one_{\{\det\gamma(f)>0\}}\mu)
\ll\sL_n.
\]
This is precisely the EID property for $f$.

For a general $f\in\mathcal{F}^n$, the currents $T_{f,f_i}$ need not
be normal because the components $f_i$ need not belong to
$D(A)$. The full proof uses an additional approximation
argument based on the density of $D(A)$ in the Hilbert space
$(\mathcal{F},\mathcal{E}_1)$, together with continuity properties of energy
measures. This extends the conclusion from $D(A)^n$ to all of
$\mathcal{F}^n$.

The preceding construction is intrinsically quadratic, since it
uses bilinear energy measures and the Hilbert-space generator,
and therefore does not extend directly to nonlinear $p$-energies.
In the nonlinear setting, the De Philippis--Rindler theorem is
instead used indirectly to detect a missing infinitesimal
direction for a singular image measure; suitable Lipschitz
approximations in this direction, together with the chain rule
and weak lower semicontinuity of the energy measures, then
contradict the assumed independence of the coordinate functions \cite{ErikssonBiqueMuruganEID}. 
 
\subsection{Extension to regular Dirichlet forms and applications}
The EID property also extends to all strongly local, regular
Dirichlet forms. There is, however, an important difference
from the setting of Dirichlet structures. If $(\mathcal{E},\mathcal{F})$ is a
strongly local, regular Dirichlet form on $L^2(X,m)$, its energy
measures need not be absolutely continuous with respect to
$m$; indeed, on many fractals they are singular with respect
to $m$. Consequently, a carr\'e du champ density relative to
the reference measure need not exist.

To overcome this issue, let $\nu$ be a minimal energy dominant
measure and, for
$f=(f_1,\ldots,f_n)\in\mathcal{F}^n$, define
\[
\gamma_\nu(f)
:=
\left(
\frac{d\Gamma(f_i,f_j)}{d\nu}
\right)_{1\leq i,j\leq n}.
\]
Since functions in $\mathcal{F}$ are initially defined only up to
$m$-null sets, while $\nu$ need not be absolutely continuous
with respect to $m$, we use quasicontinuous representatives of
the components of $f$. Such representatives are uniquely
defined outside a set of zero capacity, and energy dominant
measures do not charge such sets. The resulting image measure
is therefore well defined.

Any two minimal energy dominant measures are mutually
absolutely continuous. Moreover, changing the minimal energy
dominant measure multiplies the matrix $\gamma_\nu(f)$ by a
strictly positive scalar function. Thus the set on which
$\det\gamma_\nu(f)>0$, as well as the measure class of the
corresponding image measure, is independent of the choice of
$\nu$. A modification of the proof of Theorem~\ref{t:EID} gives the
following extension \cite{ErikssonBiqueMuruganEID}.
\begin{theorem}
	Let $(\mathcal{E},\mathcal{F})$ be a strongly local, regular Dirichlet form
	on $L^2(X,m)$, and let $\nu$ be a minimal energy dominant
	measure. For every $n\in\bN$ and
	$f=(f_1,\ldots,f_n)\in\mathcal{F}^n$, using quasicontinuous
	representatives of the components of $f$, we have
	\[
	f_*\bigl(
	\one_{\{\det\gamma_\nu(f)>0\}}\,\nu
	\bigr)
	\ll\sL_n.
	\]
\end{theorem}

This statement is reminiscent of the theorem of
De Philippis, Marchese, and Rindler discussed in Lecture~2.
For an $n$-dimensional differentiability chart $(U,\phi)$,
their result gives
\[
\phi_*(\one_U\,m)\ll\sL_n.
\]
Since $\phi$ is Lipschitz and the image has positive
$n$-dimensional Lebesgue measure, one obtains
\[
n\leq\dim_H(X,d).
\]
Taking the supremum over the chart dimensions yields the
bound of analytic dimension by Hausdorff dimension.

The same argument cannot be applied directly to martingale
dimension in the sub-Gaussian setting, because the functions
in the domain of the Dirichlet form need not be Lipschitz.
In fact, recent work of Anttila, Eriksson-Bique,
and Shimizu provides examples satisfying sub-Gaussian heat
kernel estimates with walk dimension $\beta>2$ for which
$\mathcal{F}$ contains no non-constant Lipschitz functions \cite[Theorem 1.9]{AES}.

Heat kernel regularity nevertheless supplies a sufficiently
large family of H\"older continuous functions. A consequence
of the results of Barlow, Grigor'yan, and Kumagai \cite{BGK}
is the following.

\begin{proposition}
	Let $(\mathcal{E},\mathcal{F})$ be a strongly local, regular Dirichlet form
	on a metric measure space $(X,d,m)$, where $m$ is doubling.
	Suppose that its heat kernel satisfies two-sided sub-Gaussian
	estimates. Then there exists $\theta\in(0,1]$, depending only
	on the constants in the heat kernel estimates, such that
	$\mathcal{F}$ contains an $\mathcal{E}_1$-dense collection of functions
	admitting $\theta$-H\"older continuous versions.
\end{proposition}

More concretely, heat kernel regularity implies that functions
of the form
\[
P_t h,
\qquad
t>0,\quad h\in\mathcal{F}\cap L^\infty(X,m),
\]
admit $\theta$-H\"older continuous versions. Such functions
form an $\mathcal{E}_1$-dense subset of $\mathcal{F}$.

We can now combine this density with the EID property. Suppose
that the martingale dimension is at least $n$. The density of
the H\"older functions and the continuity of energy measures
under $\mathcal{E}_1$-convergence allow us to choose a
$\theta$-H\"older map
\[
f=(f_1,\ldots,f_n):X\longrightarrow\bR^n
\]
whose modified carr\'e du champ matrix has rank $n$ on a set
of positive $\nu$-measure. The EID property then implies that
\[
f_*\bigl(
\one_{\{\det\gamma_\nu(f)>0\}}\,\nu
\bigr)
\]
is a nonzero measure absolutely continuous with respect to
$\sL_n$. Consequently, $f(X)$ has positive
$n$-dimensional Lebesgue measure and hence
\[
\dim_H f(X)=n.
\]

A $\theta$-H\"older map can increase Hausdorff dimension by
at most a factor of $1/\theta$. Therefore,
\[
n=\dim_H f(X)
\leq\frac{\dim_H(X,d)}{\theta}.
\]
Taking the supremum over all such $n$ gives
\[
d_m\leq\frac{\dim_H(X,d)}{\theta}<\infty.
\]
This proves Theorem \ref{t:finite-mtg-dim}.
In the Gaussian case $\beta=2$, the results of
Koskela and Zhou \cite{KZ} provide an $\mathcal{E}_1$-dense collection
of Lipschitz functions. Thus one may take $\theta=1$ in the
dimension argument, recovering
\[
d_m\leq\dim_H(X,d).
\]
 
\section{Non-existence of measurable differentiable structures}
\label{sec:nonexistence}
 The goal of this section is to prove a new result showing that a large
 class of metric measure spaces arising in the study of diffusions on
 fractals are not Lipschitz differentiability spaces. In particular, the
 result applies to spaces carrying strongly local, regular Dirichlet
 forms that satisfy two-sided sub-Gaussian heat kernel estimates with
 walk dimension strictly greater than two.

To state this result, we recall two analytic conditions
adapted to a general space--time scaling function. A
\emph{scale function} is a continuous and strictly increasing function
\[
\Psi\colon[0,\infty)\longrightarrow[0,\infty),
\qquad \Psi(0)=0,
\]
for which there exist constants \(C_\Psi\geq1\) and
\(1<\beta_1\leq\beta_2<\infty\) such that
\[
C_\Psi^{-1}\left(\frac{R}{r}\right)^{\beta_1}
\leq
\frac{\Psi(R)}{\Psi(r)}
\leq
C_\Psi\left(\frac{R}{r}\right)^{\beta_2}
\]
whenever \(0<r\leq R<\diam(X,d)\). The basic example is
\[
\Psi(r)=r^\beta.
\]
Here \(\Psi(r)\) represents the characteristic time scale associated
with the spatial scale \(r\). The Gaussian case corresponds to
\(\beta=2\), while \(\beta>2\) is the strictly sub-Gaussian regime.

Let \((\mathcal{E},\mathcal{F})\) be a strongly local, regular Dirichlet form on
\(L^2(X,m)\), and let \(\Gamma\) denote its energy measure. We say that
the Dirichlet form satisfies the \(\Psi\)-Poincar\'e inequality,
denoted by \({\rm PI}(\Psi)\), if there exist constants
\(C_{\rm PI}>0\) and \(A\geq1\) such that
\[
\int_{B(x,r)}
\left|f-f_{B(x,r)}\right|^2\,dm
\leq
C_{\rm PI}\Psi(r)\,
\Gamma(f,f)\bigl(B(x,A r)\bigr)
\]
for every \(f\in\mathcal{F}\), \(x\in X\), and \(r>0\), where
\[
f_{B(x,r)}
=
\frac{1}{m(B(x,r))}
\int_{B(x,r)}f\,dm.
\]
Thus \({\rm PI}(\Psi)\) controls the oscillation of a function on a
ball by its energy on a slightly larger ball, with the factor
\(\Psi(r)\) accounting for the relevant space--time scaling.

We next recall the corresponding capacity upper bound. If \(K\subset
U\), where \(U\) is open, define the relative capacity by
\[
\operatorname{Cap}_{U}(K)
:=
\inf\left\{
\mathcal{E}(u,u):
u\in\mathcal{F}_U,\ 
\widetilde u\geq1
\text{ quasi-everywhere on }K
\right\},
\]
where \(\widetilde u\) denotes a quasicontinuous representative and
\[
\mathcal{F}_U
=
\left\{
u\in\mathcal{F}:
\widetilde u=0
\text{ quasi-everywhere on }X\setminus U
\right\}.
\]
We say that \({\rm Cap}(\Psi)_{\leq}\) holds if there exist \(A>1\)
and \(C_{\rm Cap}>0\) such that
\[
\operatorname{Cap}_{B(x,Ar)}\bigl(B(x,r)\bigr)
\leq
C_{\rm Cap}\frac{m(B(x,r))}{\Psi(r)}
\]
for every \(x\in X\) and every relevant radius \(r>0\). Equivalently,
up to changing the constants, one can require the existence of a
cutoff function \(u\in\mathcal{F}\) satisfying
\[
0\leq u\leq1,\qquad
u=1\ \text{on }B(x,r),\qquad
u=0\ \text{outside }B(x,Ar),
\]
and
\[
\mathcal{E}(u,u)
\leq
C_{\rm Cap}\frac{m(B(x,r))}{\Psi(r)}.
\]
The particular value of the enlargement factor \(A>1\) is immaterial:
if the capacity upper bound holds for one such \(A\), then a standard
covering argument, together with volume doubling and the scaling
properties of \(\Psi\), shows that it holds for every \(A>1\), after
suitably adjusting \(C_{\rm Cap}\).

These conditions arise naturally from heat kernel estimates. In the
setting of strongly local regular Dirichlet forms, two-sided heat
kernel estimates with scale function \(\Psi\) imply
\({\rm PI}(\Psi)\) and the capacity upper bound
\({\rm Cap}(\Psi)_{\leq}\); see \cite{BGK}. Under the standing
assumptions, the converse was a long-standing open problem and was
recently established by the first-named author \cite{EB-res}. The
sub-Gaussian estimates with walk dimension \(\beta\) considered
earlier correspond to the scale function
\[
\Psi(r)=r^\beta.
\]
For this choice,
\[
\frac{\Psi(\epsilon r)}
{\Psi(r)\epsilon^2}
=
\epsilon^{\beta-2}.
\]
Consequently,
\[
\lim_{\epsilon\to0}\limsup_{r\downarrow0}
\frac{\Psi(\epsilon r)}
{\Psi(r)\epsilon^2}
=0
\]
whenever \(\beta>2\), whereas the expression is identically equal to
\(1\) in the Gaussian case \(\beta=2\). The following theorem
therefore gives a sharp contrast with the Gaussian situation
discussed in Lecture~2.

\begin{theorem}\label{t:diff-space}
	Let \((\mathcal{E},\mathcal{F})\) be a strongly local, regular Dirichlet form
	on \(L^2(X,m)\), where \((X,d,m)\) is a metric measure space
	and \(m\) is doubling. Suppose that \((\mathcal{E},\mathcal{F})\) satisfies
	\({\rm PI}(\Psi)\) and \({\rm Cap}(\Psi)_{\leq}\), and that
	\[
	\lim_{\epsilon\to0}\limsup_{r\downarrow0}
	\frac{\Psi(\epsilon r)}
	{\Psi(r)\epsilon^2}
	=0.
	\]
	Then \((X,d,m)\) is not a Lipschitz differentiability space.
\end{theorem}

In particular, if the Dirichlet form satisfies two-sided
sub-Gaussian heat kernel estimates with walk dimension \(\beta>2\),
then the associated metric measure space does not admit a measurable
differentiable structure. Thus the Gaussian case \(\beta=2\), in
which martingale and analytic dimensions can be compared, is
fundamentally different from the strictly sub-Gaussian case.

	\subsection{A consequence of the Poincar\'e inequality}
 
     Throughout, we assume that $(X,d)$ is a complete, separable,
     locally compact metric space and that $m$ is a Radon measure on it that satisfies $0<m(B(x,r))<\infty$ for all $x \in X, r>0$.   We also assume that $(X,d,m)$ satisfies the doubling property: that is there exists $D>0$ such that 
	\[
	m(B(x,2r)) \le D m(B(x,r)), \quad \mbox{for all $x \in X, r>0$.}
	\]
	We refer to $D$ as a \emph{doubling constant} of $m$.
	Since $(X,d,m)$ is a doubling metric measure space, the Lebesgue
	differentiation theorem \cite[Theorem 1.8]{Hei} implies that, for every
	$f\in L^1_{\operatorname{loc}}(X,m)$,
	\[
	\widetilde{f}(x):=
	\begin{cases}
		\displaystyle
		\lim_{r\downarrow0}
		\frac{1}{m(B(x,r))}
		\int_{B(x,r)}f(y)\,dm(y),
		& \text{if the limit exists},\\[2mm]
		0, & \text{otherwise},
	\end{cases}
	\]
	defines a measurable function satisfying $\widetilde{f}=f$
	$m$-almost everywhere. We call $\widetilde{f}$ the
	\emph{precise representative} of $f$, and throughout we identify
	$f\in L^1_{\operatorname{loc}}(X,m)$ with its precise representative.
	
	We say that $x\in X$ is a \emph{Lebesgue point of $f$} if
	\[
	\lim_{r\downarrow0}
	\frac{1}{m(B(x,r))}
	\int_{B(x,r)}|f(y)-f(x)|\,dm(y)=0.
	\]

	For a Radon measure $\nu$ on $X$, let $M_r\nu(x)=\sup_{s\in (0,r)} \frac{\nu(B(x,s))}{m(B(x,s))}.$ The proof of the following lemma is based on the maximal-function
	telescoping argument of Haj{\l}asz and Koskela \cite{HK00}, adapted here
	to the Poincar\'e inequality with scale function $\Psi$.
 
	\begin{lemma}\label{lem:PIptwise}
		Assume that the metric measure Dirichlet space satisfies $\PI(\Psi)$
		and volume doubling. Then there exist constants $A\geq 1$ and
		$C_{PI}>0$ such that the following holds for every $f\in\mathcal{F}$,
		represented by its precise representative.
		
		If, for some $x\in X$ and $r>0$,
		\[
		M_r\Gamma(f,f)(x)<\infty,
		\]
		then $x$ is a Lebesgue point of $f$.
		
		Moreover, if $x,y\in X$ satisfy
		\[
		M_{A d(x,y)}\Gamma(f,f)(x)<\infty
		\quad\text{and}\quad
		M_{A d(x,y)}\Gamma(f,f)(y)<\infty,
		\]
		then
		\[
		\begin{aligned}
&|f(x)-f(y)|\\
&\quad\leq C_{PI}\Psi(d(x,y))^{1/2}\Bigl(
\left(M_{A d(x,y)}\Gamma(f,f)(x)\right)^{1/2}\\
&\hspace{110pt}+\left(M_{A d(x,y)}\Gamma(f,f)(y)\right)^{1/2}\Bigr).
\end{aligned}
		\]
	\end{lemma}
	
	\begin{proof}
		We first prove the assertion concerning Lebesgue points. Suppose that
		$M_r\Gamma(f,f)(x)<\infty$, and choose $R>0$ sufficiently
		small that $AR<r$. Set
		$B_i:=B(x,2^{-i}R), \qquad i\geq 0.$
		By the Poincar\'e inequality and volume doubling,
		\[
		|f_{B_i}-f_{B_{i+1}}|
		\leq
		C\Psi(2^{-i}R)^{1/2}
		\left(M_r\Gamma(f,f)(x)\right)^{1/2}.
		\]
		The scaling properties of $\Psi$ imply that
		\[
		\sum_{i=0}^{\infty}\Psi(2^{-i}R)^{1/2}
		\leq C\Psi(R)^{1/2},
		\]
		and hence $(f_{B_i})_{i\geq0}$ is Cauchy.
		
		If $2^{-i-1}R<s\leq2^{-i}R$, another application of the Poincar\'e
		inequality and volume doubling gives
		\[
		|f_{B(x,s)}-f_{B_i}|
		\leq
		C\Psi(2^{-i}R)^{1/2}
		\left(M_r\Gamma(f,f)(x)\right)^{1/2}.
		\]
		Therefore $\lim_{s\downarrow0}f_{B(x,s)}$ exists. Since $f$ is
		represented by its precise representative, this limit equals $f(x)$.
		Furthermore,
		\[
		\frac{1}{m(B(x,s))}\int_{B(x,s)}|f-f_{B(x,s)}|\,dm
		\leq
		C\Psi(s)^{1/2}
		\left(M_r\Gamma(f,f)(x)\right)^{1/2}
		\longrightarrow0.
		\]
		Together with $f_{B(x,s)}\to f(x)$, this yields
		\[
		\frac{1}{m(B(x,s))}\int_{B(x,s)}|f-f(x)|\,dm\longrightarrow0,
		\]
		so $x$ is a Lebesgue point of $f$.
		
		We next prove the pointwise estimate. Set $R=d(x,y)$. By the preceding
		telescoping argument,
		\[
		|f(x)-f_{B(x,R)}|
		\leq
		C\Psi(R)^{1/2}
		\left(M_{AR}\Gamma(f,f)(x)\right)^{1/2},
		\]
		and similarly,
		\[
		|f(y)-f_{B(y,R)}|
		\leq
		C\Psi(R)^{1/2}
		\left(M_{AR}\Gamma(f,f)(y)\right)^{1/2}.
		\]
		After increasing $A$ by a fixed factor if necessary, the Poincar\'e
		inequality and volume doubling also give
		\[
		|f_{B(x,R)}-f_{B(y,R)}|
		\leq
		C\Psi(R)^{1/2}
		\left(M_{AR}\Gamma(f,f)(x)\right)^{1/2}.
		\]
		Combining the preceding three estimates and increasing the constant
		gives
		\[
		|f(x)-f(y)|
		\leq
		C_{PI}\Psi(R)^{1/2}
		\left(
		\left(M_{AR}\Gamma(f,f)(x)\right)^{1/2}
		+
		\left(M_{AR}\Gamma(f,f)(y)\right)^{1/2}
		\right).
		\]
	\end{proof}

	\subsection{Minimal $\ast$-upper gradients}
  The proof of Theorem \ref{t:diff-space} uses some of the main tools of the theory of differentiability spaces: curve fragments and $\ast$-upper gradients, which we now explain. 
	
	A \emph{curve fragment} in $X$ is a bi-Lipschitz map
	$\gamma:\dom(\gamma)\to X$, where $\dom(\gamma)$ is a compact subset
	of $\mathbb{R}$. Let $\Gamma(X)$ denote the collection of all curve
	fragments in $X$. For $\gamma\in\Gamma(X)$, set
	\[
	K:=\dom(\gamma),\qquad a:=\min K,\qquad b:=\max K,
	\]
	and write
	\[
	[a,b]\setminus K=\bigcup_{i\in I}(a_i,b_i),
	\]
	where the intervals $(a_i,b_i)$ are the connected components of
	$[a,b]\setminus K$. These intervals are called the \emph{gaps} of
	$\gamma$, and we define
	\[
	\gap(\gamma):=\sum_{i\in I}(b_i-a_i).
	\]
	For disjoint sets $E,F\subset X$, let $\Gamma(E,F)$ denote the
	collection of all $\gamma\in\Gamma(X)$ such that
	\[
	\gamma(\min\dom(\gamma))\in E
	\quad\text{and}\quad
	\gamma(\max\dom(\gamma))\in F.
	\]
		
	For a curve fragment $\gamma:K\to X$, we define the \emph{metric derivative} as
	$$
	d_\gamma(t):=\lim_{\substack{s\in K\\s\to t}} \frac{d(\gamma(s),\gamma(t))}{|s-t|},
	$$
	which exists at a.e. nonisolated point $t\in \dom(\gamma)$. If $g:X\to [0,\infty]$ is a Borel function, we define
	$$
	\int_\gamma g\,ds := \int_{\dom(\gamma)} g(\gamma(t))d_\gamma(t)\,dt.
	$$
	We say that a curve fragment is \emph{parameterized by unit speed} if $d_\gamma(t)=1$ for a.e. $t\in \dom(\gamma)$ and
	$$
	d(\gamma(a_i),\gamma(b_i))=|a_i-b_i|
	$$	
	for every gap $(a_i,b_i)$, $i\in I$. A unit-speed parametrization can be obtained by embedding $X$ isometrically into $\ell^\infty(X)$ via the Kuratowski embedding \cite[Exercise 12.4]{Hei}, extending the curve piecewise linearly, and applying classical results for curves in \cite[Lemma 1.1.4]{AGS08}; see  \cite[Lemma 2.15]{EBchar} for details. More precisely, it is shown in \cite[Lemma 2.15]{EBchar} that there exists a surjective Lipschitz reparametrization
	$$
	\sigma:\dom(\gamma)\to \dom(\widetilde{\gamma})
	$$
	such that $\widetilde{\gamma}\circ \sigma(t)=\gamma(t)$ for all $t\in \dom(\gamma)$ and $d_{\widetilde{\gamma}}(\sigma(t))=1$ for a.e. $t\in \dom(\gamma)$. The surjectivity of $\sigma$ is established in the proof. Since $\sigma$ is Lipschitz, it maps sets of Lebesgue measure zero to sets of Lebesgue measure zero, and hence $d_{\widetilde{\gamma}}(s)=1$ for a.e. $s\in \dom(\widetilde{\gamma})$. We also recall the standard fact that the above integrals are independent of the choice of parametrization \cite[Lemma 2.15-(5)]{EBchar}.  
	Also, let 
	\[
	\len(\gamma):=\sup_{\stackrel{n\in \N, t_1,\dots, t_n \in K}{t_1\leq \cdots \leq t_n}} \sum_{i=1}^{n-1} d(\gamma(t_{i+1}),\gamma(t_i))
	\]
	be the length of a curve fragment. In \cite[Lemma 2.15-(2),(3)]{EBchar} it is shown that $\len(\gamma)=\len(\widetilde{\gamma})$ and $\gap(\widetilde{\gamma})=\gap(\gamma)$. It follows from \cite[Proof of Lemma 2.15]{EBchar} that $\len(\widetilde{\gamma})=\gap(\widetilde{\gamma})+\int_{\widetilde{\gamma}} 1\, ds$. Thus, for any curve fragment $\gamma \in \Gamma(X)$,
	\begin{equation}
		\len(\gamma)=\gap(\gamma)+\int_\gamma 1 ds.
	\end{equation}
	For any curve fragment $\gamma \in \Gamma(X)$, we get from this 
	\begin{equation}    \label{eq:gammaint} d(\gamma(\max(\dom(\gamma))),\gamma(\min(\dom(\gamma))))\leq \len(\gamma) \leq \gap(\gamma)+\int_\gamma 1\, ds.
	\end{equation}
	
	We need the following lemma for covering a curve fragment parameterized by unit speed.
	
	\begin{lemma} \label{lem:coverfrag} Let $\delta\in(0,1)$, let $\gamma\in\Gamma(X)$ be parameterized
		by unit speed, and let $g\colon X\to[c,M]$ be a bounded lower
		semicontinuous function, where $M>c>0$. For each $n\in \N$, there exists a collection of intervals $\mathcal{I}_n$ in $\R$ with pairwise disjoint interiors such that the following hold. 
		\begin{enumerate}
			\item The endpoints of each $I\in \mathcal{I}_n$ lie in $\dom(\gamma)$.
			\item For each $I=[a,b]\in \mathcal{I}_n$, we have
			\[
			2^{-n-1}\leq d(\gamma(b),\gamma(a)) \leq |b-a|\leq 2^{-n}.
			\]
			
			\item For each $I\in \mathcal{I}_n$ we have $\gap(\gamma|_{I\cap \dom(\gamma)})\leq \delta2^{-n-4}$.
			\item For each $I=[a,b]\in \mathcal{I}_n$
			\[
			\int_{\gamma|_{I\cap \dom(\gamma)}} g ds \leq 2^{1-n}\inf_{p\in B(\gamma(a), 2^{-n})} g(p).
			\]
			\item 
			\[
			\lim_{n\to \infty}\lambda(\dom(\gamma) \setminus \bigcup_{I\in \mathcal{I}_n} I)=0.
			\]
		\end{enumerate}
		
	\end{lemma}
	\begin{proof} Define $K:=\dom(\gamma)$.
		Let $\widetilde{K}_n$ be the set of $t\in K$ such that the following hold: 
		\begin{enumerate}[(a)]
			\item 
			\begin{equation} \label{e:1cond}
				\frac{\lambda([t,t+2^{-n}]\setminus K)}{2^{-n}}\leq 2^{-4}\delta.
			\end{equation}
			\item  For all $s\in K\cap (t-2^{-n},t+2^{-n}) \setminus \{t\}$ 
			\begin{equation} \label{e:2cond}
				\frac{d(\gamma(s),\gamma(t))}{|t-s|}\geq (1-2^{-2}).
			\end{equation}
			\item 
			\begin{equation} \label{e:3cond}
				\int_{\gamma|_{[t, t+2^{-n}]\cap \dom(\gamma)}} g\, ds \leq 2^{1-n}\inf_{p\in B(\gamma(t), 2^{-n})} g(p).
			\end{equation}
		\end{enumerate}
		We claim that the set $\widetilde K_n.$ is closed. To see this, assume that $t_i\in \widetilde K_n.$ are such that $\lim_{i\to \infty} t_i=t$. By the dominated convergence theorem   
		\[
		\lim_{i\to \infty } \frac{\lambda([t_i,t_i+2^{-n}]\setminus K)}{2^{-n}}=\frac{\lambda([t,t+2^{-n}]\setminus K)}{2^{-n}}
		\]
		and  hence $t$ satisfies (a). Now, assume that $t$ fails (b). Then there exists $s\in K \cap (t-2^{-n}, t+2^{-n}) \setminus \{t\}$ for which the opposite inequality would hold strictly. Since $\lim_{i\to \infty} t_i=t$,  $s\in K \cap (t_i-2^{-n}, t_i+2^{-n})$ for all large enough $i$ and 
		\[
		\lim_{i\to \infty } \frac{d(\gamma(s),\gamma(t_i))}{|t_i-s|} =  \frac{d(\gamma(s),\gamma(t))}{|t-s|} < (1-2^{-2})
		\]
		implies that $t_i$ doesn't satisfy (b) for large enough $i$, which is a contradiction. Thus $t$ satisfies (b). Finally, for (c), notice that the function $$t \mapsto 	\int_{\gamma|_{[t, t+2^{-n}]\cap \dom(\gamma)}} g\, ds$$ is continuous in $t$ by the estimate \[
		\begin{aligned}
&\abs{\int_{\gamma|_{[t, t+2^{-n}]\cap \dom(\gamma)}} g\,ds-
\int_{\gamma|_{[t', t'+2^{-n}]\cap \dom(\gamma)}} g\,ds}\\
&\qquad\le 2\sup\abs{g}\abs{t-t'}\le 2M\abs{t-t'}.
\end{aligned}\]
		On the other hand, the function $t \mapsto \inf_{p\in B(\gamma(t), 2^{-n})} g(p)$ is upper-semicontinuous in $t$, since the function $x \mapsto \inf_{p \in B(x,2^{-n})} g(p)$ is upper-semicontinuous due to $B(x,2^{-n})$ being open. This implies that the set  $t \in K$ satisfying \eqref{e:3cond} is a closed set, which concludes the proof of the  claim that the set $\widetilde{K}_n$ is closed.
		
		Next, we claim that
		\begin{equation}  \label{e:Kn-meas}
			\lim_{n\to \infty} \lambda(K\setminus \widetilde{K}_n)=0.
		\end{equation}
		Lebesgue differentiation implies that 
		\[
		\lim_{n \to \infty} \frac{\lambda([t,t+2^{-n}] \setminus K)}{2^{-n}}=0, \quad \mbox{for a.e.~$t \in K$.}
		\]
		The choice of  unit speed parametrization for $\gamma$   imply that 
		\[
		\lim_{s \to t}  	\frac{d(\gamma(s),\gamma(t))}{|t-s|}= 1, \quad \mbox{for a.e.~$t \in K$.}
		\]
		Lebesgue differentiation and the lower semicontinuity of $g$ implies that 
		\[
		\begin{aligned}
\lim_{n\to\infty}2^n\int_{\gamma|_{[t,t+2^{-n}]\cap\dom(\gamma)}}g\,ds
&=g(\gamma(t))\\
&=\lim_{n\to\infty}\inf_{p\in B(\gamma(t),2^{-n})}g(p),
\end{aligned}
\]
for a.e.~$t\in K$.
		Hence for a.e.~$t \in K$, there exists $n \in \mathbb{N}$ such that $t \in \bigcap_{l \ge n}\widetilde{K}_l$. Thus \eqref{e:Kn-meas} follows from
		\begin{align*}
			\limsup_{n\to \infty} \lambda(K \setminus \widetilde{K}_n) &\le \lim_{n \to \infty} \lambda\left(K \setminus \bigcap_{l \ge n}\widetilde{K}_l\right)\\
			&= \lambda \left(\bigcap_{n=1}^\infty \left(K \setminus \bigcap_{l \ge n}\widetilde{K}_l\right) \right)\\
&=\lambda\left(K\setminus\bigcup_{n=1}^\infty\bigcap_{l\ge n}\widetilde{K}_l\right)=0.
		\end{align*}

	 If $\widetilde K_n=\varnothing$, set $\mathcal I_n=\varnothing$,
	 and no further construction is needed. For the remainder, assume that
	 $\widetilde K_n\neq\varnothing$ and let
	 $a_1=\min(\widetilde K_n)$. Since $\widetilde K_n$ is closed,
	 $a_1\in\widetilde K_n$. Hence by \eqref{e:1cond}, there exists a $b_1\in K\cap [a_1+3\cdot 2^{-n-2},a_1+2^{-n})$. Set $I_1=[a_1,b_1]$. We next describe the recursive construction of $a_{i+1},b_{i+1}, I_{i+1}$ 
		assuming $a_i,b_i,I_i=[a_i,b_i]$ have been defined such that $a_i \in \widetilde{K}_n$ and $b_i \in K$.
		If $\max(\widetilde{K}_n)< b_i$, we stop the recursive construction and set 
		\[
		\mathcal{I}_n:= \{[a_j,b_j]: j=1,\ldots,i\}.
		\]

		Next, we consider the case $\max(\widetilde{K}_n) \ge b_i$.
		In this case, let $a_{i+1}=\inf(\widetilde{K}_n\cap [b_i,\infty))$. 
		Since $\widetilde{K}_n$ is closed, $a_{i+1} \in \widetilde{K}_n$. By \eqref{e:1cond}, there exists a $b_{i+1}\in K\cap [a_{i+1}+3\cdot 2^{-n-2},a_{i+1}+2^{-n})$. Set $I_{i+1}=[a_{i+1},b_{i+1}]$.  
		
		By compactness of $K$, and since each interval $I_i$ has length at least $3\cdot2^{-n-2}$, this recursion stops after finitely many steps   and yields a collection $\mathcal{I}_n=\{I_1,\dots, I_N\}.$ Since the intervals were constructed with increasing end points, the intervals have pairwise disjoint interiors. We now show the following properties for $i=1,\dots, N$:
		\begin{enumerate}[(i)]
			\item $a_i\in \widetilde{K}_n$, $b_i > a_i$ and $b_i\in K$.
			\item   $I_i=[a_i,b_i]$ satisfies all the properties (1)-(5) in the statement of the Lemma.
		\end{enumerate}
		The properties (i) and (1) are immediate from the construction and the fact that $\widetilde{K}_n \subset K$. For verifying (2), note that
		\[
		d(\gamma(b_i),\gamma(a_i))\stackrel{\eqref{e:2cond}}{\geq}  (1-2^{-2})(b_i-a_i) \geq 3(1-2^{-2})2^{-n-2} \geq 2^{-n-1}.
		\]
		Since $a_i\in \widetilde{K}_n$ and $a_i<b_i \leq a_i+2^{-n}$, we have
		\[
		\frac{\lambda([a_i,b_i]\setminus K)}{2^{-n}}\leq \frac{\lambda([a_i,a_i+2^{-n}]\setminus K)}{2^{-n}}\stackrel{\eqref{e:1cond}}{\leq} 2^{-4}\delta.
		\]
		Thus we have $\gap\bigl(\gamma|_{I_i\cap\dom(\gamma)}\bigr)\leq \delta 2^{-n-4}$, which yields $(3)$. 
		
		Now, $(4)$ follows since  $\int_{\gamma|_{I_i\cap \dom(\gamma)}} g \, ds \leq \int_{\gamma|_{[a_i, a_i+2^{-n}]\cap \dom(\gamma)}} g \, ds$ and $a_i \in \widetilde{K}_n$.
		
		It remains to show that $(5)$ holds. To show this, it suffices to observe that the construction ensures that $\mathcal{I}_n$ covers $\widetilde{K}_n$, and to apply \eqref{e:Kn-meas}.
	\end{proof}
	
	Let $f\in \LIP(X)$. We call a Borel function $g\colon X\to[0,M]$, for some
	$M>0$, a \emph{$\ast$-upper gradient of $f$} if there exists
	a constant $L>0$ such that, for every $\gamma\in\Gamma(X)$,
	\begin{equation}\label{def:astuppergrad}
		\left|f\bigl(\gamma(\min K)\bigr)
		-f\bigl(\gamma(\max K)\bigr)\right|
		\leq \int_\gamma g\,ds+L\gap(\gamma),
	\end{equation}
	where $K=\dom(\gamma)$.
	We say that $g$ is a \emph{minimal $\ast$-upper gradient} of $f$ if $g$ is a $\ast$-upper gradient of $f$ and satisfies $g'\geq g$ $m$-almost everywhere for every $\ast$-upper gradient $g'$ of $f$.  The minimal $\ast$-upper gradient is often denoted by $g_f$. By \cite[Proposition 2.10]{bate2024fragment}, a minimal $\ast$-upper gradient exists and is unique up to $m$-almost everywhere equivalence. 
	\begin{lemma}\label{lem:astupperbounded}
		Let $f \in \LIP(X)$. 
		If $g$ is an $\ast$-upper gradient of $f$, then $\min(g,\LIP[f])$ is an $\ast$-upper gradient of $f$ and \eqref{def:astuppergrad} is true with $L=\LIP[f]$.
	\end{lemma}
	\begin{proof}
		By Lebesgue differentiation, the condition that $g$ is a $\ast$-upper gradient of $f$ is equivalent to $|(f\circ \gamma)'(t)|\leq g(\gamma(t))d_\gamma(t)$ for a.e. $t\in \dom(\gamma)$ and for all $\gamma \in \Gamma(X)$. Since $|(f\circ \gamma)'(t)|\leq \LIP[f] d_\gamma(t)$, we see that $\min(g,\LIP[f])$ is always an $\ast$-upper gradient of $f$ whenever $g$ is one. 
		
		Let $K=\dom(\gamma)$. Now, if $(a_i,b_i), i\in I$ are the gaps of $\gamma$, then 
		\[
		\begin{aligned}
&f(\gamma(\max(K)))-f(\gamma(\min(K)))\\
&\quad=\int_{\dom(\gamma)}(f\circ\gamma)'(t)\,dt
+\sum_{i\in I}\left(f(\gamma(b_i))-f(\gamma(a_i))\right).
\end{aligned}
		\]
		From this, and the previous paragraph we immediately get \eqref{def:astuppergrad} with $L=\LIP[f]$.
	\end{proof}

	\begin{lemma}\label{lem:mazur}
		Assume that $g_i:X\to[0,L]$ is measurable for every $i\in\mathbb{N}$. Then, for every bounded Borel set $A\subset X$, there exists a measurable function $g:X\to[0,L]$, depending on $A$, such that
		$$
		\int_A g\,dm\leq \liminf_{i\to\infty}\int_A g_i\,dm
		$$
		and, for every $\gamma\in\Gamma(X)$,
		$$
		\int_\gamma g\,ds\geq \liminf_{i\to\infty}\int_\gamma g_i\,ds.
		$$
	\end{lemma}
	\begin{proof}
		Let $A \subset X$ be bounded. After passing to a subsequence, we may assume that $$\liminf_{i\to \infty} \int_A g_i\, dm=\lim_{i\to\infty}\int_A g_i\, dm.$$ 
		By Dunford--Pettis theorem \cite[Theorem 4.30]{Brezis} and \cite[Exercise 4.36]{Brezis}, the set $\{\restr{g_n}{A}:n \in \mathbb{N}\}$ is relatively compact in the weak topology of $L^1(A,m|_A)$.
		By using Mazur's lemma \cite[Corollary 3.6, Exercise 3.4-(1)]{Brezis} and taking convex combinations $\tilde{g}_k=\sum_{i=k}^{N_k} \alpha_{k,i} g_i$, we can ensure that $\restr{\tilde{g}_k}{A}$ converges in $L^1(A,m|_A)$. By passing to a further subsequence   if necessary \cite[Theorem 4.9]{Brezis}, we may assume that $\tilde{g}_k$ converge pointwise $m$-a.e.~in $A$. Define $g:X \to [0,L]$ by
		$$g(x)=
		\begin{cases}
			\displaystyle \lim_{k\to\infty}\widetilde g_k(x),
			& \text{if }x\in A\text{ and the limit exists},\\
			L, & \text{otherwise}.
		\end{cases}$$
		Since $\tilde{g}_k \le L$ everywhere, we have
		\[
		0 \le L - g(x) \le \liminf_{k \to \infty} (L-\tilde{g}_k(x)), \quad  \mbox{for all $x \in X$.}
		\]
		Therefore by Fatou's lemma and linearity of integrals, we have 
		\[
		\int_\gamma g\,ds \ge \limsup_{k \to \infty} \int_\gamma \tilde{g}_k\,ds \ge \liminf_{i \to \infty} \int_\gamma g_i\,ds, \quad \mbox{for all $\gamma \in \Gamma(X)$.} 
		\]
		Since $\tilde{g}_k$ converges in $L^1(A,m|_A)$ to $g$, we have
		\[
		\begin{aligned}
\int_A g\,dm
&=\lim_{k\to\infty}\int_A\tilde{g}_k\,dm\\
&=\lim_{k\to\infty}\sum_{i=k}^{N_k}\alpha_{k,i}\int_A g_i\,dm\\
&=\lim_{i\to\infty}\int_A g_i\,dm
=\liminf_{i\to\infty}\int_A g_i\,dm.
\end{aligned} 
		\]
	\end{proof}
	
	We adapt a definition of connectivity from \cite[Definition 3.1]{EBchar} and \cite[Lemma 3.5]{LiBate}. In these references, such a condition arose in the context of proving Poincar\'e inequalities for spaces supporting a stronger form of differentiation.
	\begin{definition} \label{d:connect} Let $\delta,\epsilon \in (0,1)$. We say that an annulus $B(x,2r)\setminus B(x,r)$ is $(\delta,\epsilon)$-connected if for every Borel set $E\subset B(x,2r)$  with $m(E)/m(B(x,2r))<\epsilon$ there exists a $\gamma\in \Gamma(B(x,r), X\setminus B(x,2r))$ with $\gamma(\dom(\gamma)) \cap E\subset \{\gamma(\min(\dom(\gamma))\}$ and
		\[
		\gap(\gamma)\leq \delta r.
		\]
	\end{definition}

	The argument in the following lemma is inspired by the proof that positive modulus implies a lower bound for conformal dimension. The original argument is due to Pansu and Bourdon \cite{pansu, bourdon}, but the presentation here is closer to that of \cite{Tyson} see also \cite[Theorem 4.1]{bhw}. In these proofs, one can find a sequence of functions with small $L^p$-norm, and extract an estimate for modulus by taking a limit of such functions. Here, the role of modulus is replaced by the $\ast$-upper-gradient
	property, while poorly connected annuli provide the deformations needed
	in the argument. These deformations are used to modify the minimal
	$\ast$-upper gradient and construct a sequence of $\ast$-upper gradients
	that ultimately contradicts its minimality.
	
	\begin{lemma} \label{l:zero-ug}  Let $(X,d)$ be a complete separable metric space, and let $m$ be a doubling measure on $X$ with doubling constant $D$. Set $C_D:=2^9D^3$. Let $f\in\LIP(X)$, and suppose that its minimal $\ast$-upper gradient $g_f$ satisfies
		\begin{equation} \label{e:nonzero-gf}
			m(\{x\in X:g_f(x)>0\})>0.
		\end{equation}
		Then, for every $\delta\in(0,1)$ and $r_0>0$, there exist $x\in X$ and $r\in(0,r_0)$ such that the annulus $B(x,2r)\setminus B(x,r)$
		is $(\delta,\delta/C_D)$-connected.
	\end{lemma}
	
	\begin{proof}  
		For each $n\in\N$, let $N_n$ be a $2^{-n}$-net of $X$; that is, a maximal subset of $X$ such that any two distinct points $x,y\in N_n$ satisfy $d(x,y)\geq 2^{-n}$.
		
		Suppose, to the contrary, that there exist $\delta\in(0,1)$ and $r_0>0$ such that every annulus $B(x,2r)\setminus B(x,r)$, with $x\in X$ and $r\in(0,r_0)$, is not $(\delta,\delta/C_D)$-connected. We will obtain a contradiction by showing that $g_f=0$ $m$-a.e. By replacing the metric $d$ with $\alpha d$, where $\alpha:=2r_0^{-1}\in(0,\infty)$, we may assume that $r_0=2$. Note that this rescaling preserves the doubling constant $D$ of $m$, while the minimal $\ast$-upper gradient $g_f$ is replaced by $\alpha^{-1}g_f$. By \eqref{e:nonzero-gf}, $\Lip(f) \neq 0$.  Replacing $f$ by $\Lip(f)^{-1}f$, we may assume that $f$ is $1$-Lipschitz. Note that the rescaled function still satisfies \eqref{e:nonzero-gf}.
		By Lemma \ref{lem:astupperbounded}, and after possibly modifying $g_f$ on a set of measure zero, we may and will assume that $g_f\leq 1$.

		Since $(X,d)$ is separable, \eqref{e:nonzero-gf} implies that there exists $x_0 \in X$ such that $\int_{B(x_0,r)} g_f dm >0$ for all $r>0$.  Since $m$ is finite on balls, $m(\partial B(x_0,r))>0$ for at most countably many $r>0$. We may therefore choose $r\in(0,r_0)$ such that $m(\partial B(x_0,r))=0$.  Since $m$ is a Radon measure that is finite on all balls, $m$ is outer regular, see e.g. the proof in \cite[Theorem 1.10]{Mattila95}.  Since $X$ is separable and $m$ is outer regular, we may apply the Vitali--Carath\'eodory theorem \cite[p.~108, Chapter 4]{HKST} to the function $\restr{g_f}{\overline{B(x_0,r)}}$ on the metric measure space $(\overline{B(x_0,r)},m|_{\overline{B(x_0,r)}})$. Thus, after truncating and choosing $\epsilon>0$ sufficiently small, there exists a lower semicontinuous function $\widetilde{g}:\overline{B(x_0,r)}\to[\epsilon,1]$ such that $\widetilde{g}\geq g_f$ on $\overline{B(x_0,r)}$ and
		\[		\int_{B(x_0,r)}\widetilde{g}\,dm	<	\frac32\int_{B(x_0,r)}g_f\,dm.
		\]
		Define $g:X\to[\epsilon,1]$ by
		\[
		g(x):=
		\begin{cases}
			\widetilde{g}(x), & x\in\overline{B(x_0,r)},\\
			1, & x\in X\setminus\overline{B(x_0,r)}.
		\end{cases}
		\]
		Then $g$ is lower semicontinuous and satisfies $g_f\leq g$ everywhere. Since $m(\partial B(x_0,r))=0$, we have
		$\lim_{n \to \infty}\int_{B(x_0,r+2^{-n})}g\,dm=
		\int_{B(x_0,r)}g\,dm$.
		Consequently, for all sufficiently large $n$,
		\begin{equation}\label{eq:intbound}
			\int_{B(x_0,r+2^{-n})}g\,dm
			<
			2\int_{B(x_0,r)}g_f\,dm.
		\end{equation}
		
		For each $n \in \mathbb{N}$ and  $p\in N_n$,   the failure of the $(\delta,\delta/C_D)$-connectivity condition of the annulus $B(p,2^{1-n}) \setminus B(p,2^{-n})$   yields a set $E_p\subset B(p,2^{1-n})$ satisfying  
		\begin{equation} \label{e:cont1}
			m(E_p)/m(B(p,2^{1-n})) < \delta/C_D,
		\end{equation}
		such that, for every
		$\gamma\in \Gamma(B(p,2^{-n}), X\setminus B(p,2^{1-n}))$ satisfying $\gamma(\dom(\gamma)) \cap E_p \subset \{\gamma(\min(\dom(\gamma)))\}$, we have \begin{equation}  \label{e:cont2}
			\gap(\gamma)> \delta 2^{-n}. 
		\end{equation}  By outer regularity and the strict inequality in \eqref{e:cont1}, and after possibly enlarging $E_p$ slightly, we may moreover assume that $E_p$ is open and $E_p\subset B(p,2^{1-n})$. Note that the condition in \eqref{e:cont2} is preserved under such an enlargement.
		
		Define
		\[
		\psi_p(x):=\frac{2^6}{\delta}1_{E_p}(x),
		\qquad p\in N_n,
		\]
		and, for each $n\in\mathbb{N}$, let
		\[
		g_n(x):=
		\sum_{p\in N_n}
		\psi_p(x)\inf_{y\in B(p,2^{-n})}g(y).
		\]
		
		We claim that for each $\gamma \in \Gamma(X)$
		\begin{equation}\label{eq:intgoal}
			\int_\gamma g \,ds \leq \liminf_{n\to \infty} \int_\gamma g_n \,ds.
		\end{equation}
		We parametrize $\gamma$ by unit speed. Lemma \ref{lem:coverfrag} yields for every $n\in \N$ a collection  $\mathcal{I}_n$ of intervals in $\R$ with pairwise disjoint interiors such that
		\begin{enumerate}
			\item The endpoints of each $I\in \mathcal{I}_n$ lie in $\dom(\gamma)$.
			\item For each $I=[a,b]\in \mathcal{I}_n$, we have
			\[
			2^{-n-1}\leq d(\gamma(b),\gamma(a)) \leq |b-a|\leq 2^{-n}.
			\]
			\item For each $I\in \mathcal{I}_n$ we have $\gap(\gamma|_{I\cap \dom(\gamma)})\leq \delta2^{-n-4}$.
			\item For each $I=[a,b]\in \mathcal{I}_n$ 
			\[
			\int_{\gamma|_{I\cap \dom(\gamma)}} g \,ds \leq 2^{1-n}\inf_{y\in B(\gamma(a), 2^{-n})} g(y).
			\]
			\item 
			\[
			\lim_{n\to \infty}\lambda(\dom(\gamma) \setminus \bigcup_{I\in \mathcal{I}_n} I)=0
			\]
		\end{enumerate}
		
		For each $n \in \mathbb{N}$ and for each  $I=[a,b]\in \mathcal{I}_n$ let $p_I \in N_{n+3}$ be such that $\gamma(a)\in B(p_I,2^{-n-3})$. We slightly abbreviate the restrictions as follows $\gamma_I:=\gamma|_{\dom(\gamma)\cap I}$. 
		By condition (2), $\gamma_{I} \in \Gamma(B(p_I,2^{-n-3}),X\setminus B(p_I,2^{-n-2}))$ since $d(p_I,\gamma(a))<2^{-n-3}$ and $d(p_I,\gamma(b)) \ge d(\gamma(a),\gamma(b))-d(p_I,\gamma(a)) > 2^{-n-1}- 2^{-n-3} >2^{-n-2}$. We also have $B(p_I, 2^{-n-3}) \subset B(\gamma(a), 2^{-n})$ for every $I= [a,b]\in \mathcal{I}_n$, since $d(\gamma(a),p_I)<2^{-n-3}$. Thus, by condition (4) above
		\begin{equation}\label{eq:glowerboundint}
			2^{1-n} \inf_{y\in B(p_I, 2^{-n-3})} g(y) \geq  2^{1-n} \inf_{y\in B(\gamma(a), 2^{-n})} g(y) \geq  \int_{\gamma_I} g\,ds.
		\end{equation}

		We next show a lower bound for $\int_{\gamma_I}1_{E_{p_I}} ds$. Consider the curve fragment $\hat{\gamma}_I=\gamma|_{\{a,b\}\cup (\dom(\gamma_I)\setminus \gamma_I^{-1}(E_{p_I}))}$, whose domain is closed since $E_{p_I}$ is open.  We have $\hat{\gamma}_I(a)\in B(p_I,2^{-n-3})$ and  $d(\hat{\gamma}_I(b),p_I)\geq d(\hat{\gamma}_I(b),\hat{\gamma}_I(a))-d(\hat{\gamma}_I(a),p_I)\geq 2^{-n-2}$. Thus $\hat{\gamma}_I(b) \not\in B(p_I,2^{-n-2})$   and hence \[
\begin{gathered}
\hat{\gamma}_I\in\Gamma(B(p_I,2^{-n-3}),X\setminus B(p_I,2^{-n-2}))
\quad\text{and}\\
\hat{\gamma}_I(\dom(\hat{\gamma}_I))\cap E_{p_I}\subset\{\gamma(a)\}.
\end{gathered}
\] Therefore by \eqref{e:cont2} applied at scale $n+3$, we have 
		\begin{equation} \label{e:gap-hat}
			\gap(\hat{\gamma}_I)> \delta2^{-n-3}. 
		\end{equation}

	By the definition of $\widehat{\gamma}_I$, its domain is obtained
	from $\dom(\gamma_I)$ by removing
	$\gamma_I^{-1}(E_{p_I})$, apart from the two endpoints. Since
	$\gamma_I$ is parameterized by unit speed, it follows that
	\[
	\gap(\widehat{\gamma}_I)
	=
	\gap(\gamma_I)
	+
	\int_{\gamma_I}\one_{E_{p_I}}\,ds.
	\]
	Consequently, \eqref{e:gap-hat} and condition \textup{(3)} give
	\[
	\delta 2^{-n-3}
	<
	\gap(\widehat{\gamma}_I)
	\leq
	\delta 2^{-n-4}
	+
	\int_{\gamma_I}\one_{E_{p_I}}\,ds.
	\]
	Therefore,
	\begin{equation}\label{e:lineint}
		\delta 2^{-n-4}
		\leq
		\int_{\gamma_I}\one_{E_{p_I}}\,ds.
	\end{equation}
		Combining \eqref{e:lineint} with \eqref{eq:glowerboundint} we get
		\begin{align*}
			\int_{\gamma_I} g_{n+3}\, ds &\geq \frac{2^6}{\delta} \inf_{y\in B(p_I,2^{-n-3})} g(y)   \int_{\gamma_I} \one_{E_{p_I}} \,ds \\
			& \stackrel{\eqref{e:lineint}}{\geq} \inf_{y\in B(p_I, 2^{-n-3})} g(y) 2^{2-n} \stackrel{\eqref{eq:glowerboundint}}{\geq } \int_{\gamma_I} g \,ds.
		\end{align*}
		Summing this estimate over $I\in\mathcal{I}_n$, we obtain
		\begin{equation} \label{e:intgn-g}
			\sum_{I\in\mathcal{I}_n}\int_{\gamma_I} g_{n+3}\,ds
			\geq
			\sum_{I\in\mathcal{I}_n}\int_{\gamma_I} g\,ds.
		\end{equation}
		
		Since the intervals in $\mathcal{I}_n$ have pairwise disjoint interiors, and $\gamma$ is parametrized by unit speed,
		$$
		\sum_{I\in\mathcal{I}_n}\int_{\gamma_I} g\,ds
		=
		\int_{\dom(\gamma)\cap\bigcup_{I\in\mathcal{I}_n}I}
		g(\gamma(t))\,dt.
		$$
		Thus, since $0\leq g\leq 1$, condition (5) implies
		\begin{equation} \label{e:limitIn}
			0\leq
			\int_\gamma g\,ds-
			\sum_{I\in\mathcal{I}_n}\int_{\gamma_I}g\,ds
			\leq
			\lambda\left(
			\dom(\gamma)\setminus\bigcup_{I\in\mathcal{I}_n}I
			\right)
			\longrightarrow 0.
		\end{equation}
		Moreover, since $g_{n+3}\geq0$ and the intervals in $\mathcal{I}_n$ have pairwise disjoint interiors, we have
		$$
		\sum_{I\in\mathcal{I}_n}\int_{\gamma_I}g_{n+3}\,ds
		\leq
		\int_\gamma g_{n+3}\,ds.
		$$
		Consequently, we obtain \eqref{eq:intgoal} since
		\begin{align*}
			\int_\gamma g\,ds &\stackrel{\eqref{e:limitIn}}{=}	\lim_{n\to\infty}
			\sum_{I\in\mathcal{I}_n}\int_{\gamma_I}g\,ds
			\stackrel{\eqref{e:intgn-g}}{\leq}
			\liminf_{n\to\infty}
			\sum_{I\in\mathcal{I}_n}\int_{\gamma_I}g_{n+3}\,ds  \\
			&\leq
			\liminf_{n\to\infty}\int_\gamma g_{n+3}\,ds 
			=
			\liminf_{n\to\infty}\int_\gamma g_n\,ds.
		\end{align*}

		We next estimate the integral of the functions $g_n$. We have
		\begin{align}
			&\int_{B(x_0,r)} g_n\, dm \nonumber\\
&\leq \sum_{\stackrel{p\in N_n}{B(p, 2^{1-n})\cap B(x_0,r)\neq \emptyset}} \inf_{y\in B(p, 2^{-n})} g(y) \int \psi_p \,dm \nonumber \\
			&\leq \sum_{\stackrel{p\in N_n}{B(p, 2^{1-n})\cap B(x_0,r)\neq \emptyset}} \frac{2^6}{\delta}m(E_p) \inf_{y\in B(p, 2^{-n})} g(y)\nonumber \\
			&\stackrel{\eqref{e:cont1}}{\leq} \sum_{\stackrel{p\in N_n}{B(p, 2^{1-n})\cap B(x_0,r)\neq \emptyset}} \frac{2^6}{C_D}m(B(p,2^{1-n})) \inf_{y\in B(p, 2^{-n})} g(y)\nonumber \\
			&\leq \sum_{\stackrel{p\in N_n}{B(p, 2^{1-n})\cap B(x_0,r)\neq \emptyset}} \frac{2^6D^2}{C_D}m(B(p,2^{-1-n})) \inf_{y\in B(p, 2^{-n-1})} g(y)\nonumber \\
			&\leq \sum_{\stackrel{p\in N_n}{B(p, 2^{1-n})\cap B(x_0,r)\neq \emptyset}} \frac{D^2 2^6 }{C_D}\int_{B(p,2^{-n-1})}  g(y) \,dm\nonumber \\
			&\leq \frac{2^6D^2 }{C_D}\int_{B(x_0,r+2^{2-n})}  g(y)\, dm \stackrel{\eqref{eq:intbound}}{\leq} \frac{2^7D^2 }{C_D}\int_{B(x_0,r)}  g_f(y)\, dm\nonumber.
		\end{align}
		Here, the final estimate holds for all large enough $n$ and we used  that $D$ is the doubling constant of $m$, that the balls $B(p,2^{-n-1})$ are pairwise disjoint and $m(\partial B(x_0,r))=0$. 
		Thus, by our choice of $C_D$ we get for all large enough $n$
		\begin{equation}\label{eq:intboundgn}
			\int_{B(x_0,r)} g_n dm \leq \frac{1}{2} \int_{B(x_0,r)} g_f dm.
		\end{equation}
		
	Let $h$ be the function provided by Lemma~\ref{lem:mazur} with
	$A=B(x_0,r)$. By \eqref{eq:intgoal} and the defining property of $h$,
	\[
	\int_\gamma h\,ds
	\geq
	\liminf_{n\to\infty}\int_\gamma g_n\,ds
	\geq
	\int_\gamma g\,ds
	\]
	for every $\gamma\in\Gamma(X)$. Since $g$ is a $\ast$-upper gradient
	of $f$, it follows that $h$ is also a $\ast$-upper gradient of $f$.
	Moreover, Lemma~\ref{lem:mazur} and \eqref{eq:intboundgn} give
	\[
	\int_{B(x_0,r)}h\,dm
	\leq
	\frac12\int_{B(x_0,r)}g_f\,dm.
	\]
	On the other hand, the minimality of $g_f$ implies that
	$g_f\leq h$ $m$-almost everywhere. This contradicts
	\[
	\int_{B(x_0,r)}g_f\,dm>0.
	\]
	\end{proof}

	The following is a  weak-type inequality for tthe maximal function.
	\begin{lemma}\label{lem:weak-maximal}
		Let $(X,d)$ be a complete separable metric space and let $m$ be a
		doubling measure with doubling constant $D$. Let $\nu$ be a positive
		Radon measure on $X$. Fix $x\in X$, $r>0$, and $L>0$, and define
		\[
		E:=\{y\in B(x,r):M_r\nu(y)\geq L\},
		\]
		where
		\[
		M_r\nu(y):=\sup_{s\in(0,r)}
		\frac{\nu(B(y,s))}{m(B(y,s))}.
		\]
		Then $E$ is Borel and
		\[
		m(E)\leq \frac{D^3}{L}\nu(B(x,2r)).
		\]
	\end{lemma}
	
	\begin{proof}
		We first note that $M_r\nu$ is lower semicontinuous. Let $y_i\to y$
		and suppose that $M_r\nu(y)>a$. Choose $s\in(0,r)$ such that
		\[
		\frac{\nu(B(y,s))}{m(B(y,s))}>a.
		\]
		Set $s_i:=s-d(y_i,y)$. For all sufficiently large $i$, we have
		$s_i\in(0,r)$ and
		\[
		B(y_i,s_i)\subset B(y,s).
		\]
		Moreover,
		\[
		1_{B(y_i,s_i)}\longrightarrow 1_{B(y,s)}
		\]
		pointwise. Hence, by dominated convergence,
		\[
		\nu(B(y_i,s_i))\longrightarrow \nu(B(y,s))
		\qquad\text{and}\qquad
		m(B(y_i,s_i))\longrightarrow m(B(y,s)).
		\]
		Therefore, for all sufficiently large $i$,
		\[
		M_r\nu(y_i)
		\geq
		\frac{\nu(B(y_i,s_i))}{m(B(y_i,s_i))}
		>a.
		\]
		Thus $M_r\nu$ is lower semicontinuous, and in particular $E$ is
		Borel.
		
		Fix $L'\in(0,L)$. For every $y\in E$, since $M_r\nu(y)\geq L>L'$,
		there exists $s_y\in(0,r)$ such that
		\[
		\nu(B(y,s_y))>L'm(B(y,s_y)).
		\]
		Apply the $5B$-covering lemma \cite[Theorem 1.2]{Hei} to the family
		\[
		\mathcal B:=\{B(y,s_y):y\in E\}.
		\]
		There exists a pairwise disjoint subcollection $\{B_i\}_{i\in I}$
		such that
		\[
		E\subset\bigcup_{i\in I}5B_i.
		\]
		Since $X$ is separable, the family $\{B_i\}_{i\in I}$ is at most
		countable. By volume doubling,
		\[
		m(5B_i)\leq m(8B_i)\leq D^3m(B_i),
		\]
		and hence
		\[
		m(E)
		\leq
		\sum_im(5B_i)
		\leq
		D^3\sum_im(B_i)
		<
		\frac{D^3}{L'}\sum_i\nu(B_i).
		\]
		Since the balls $B_i$ are pairwise disjoint and each $B_i$ is
		contained in $B(x,2r)$,
		\[
		\sum_i\nu(B_i)\leq\nu(B(x,2r)).
		\]
		Thus
		\[
		m(E)
		\leq
		\frac{D^3}{L'}\nu(B(x,2r)).
		\]
		Letting $L'\uparrow L$ gives
		\[
		m(E)\leq\frac{D^3}{L}\nu(B(x,2r)).
		\]
	\end{proof}

	\begin{lemma} \label{l:sub-gauss}
		Let $\Psi:[0,\infty) \to [0,\infty)$ be a scale function such that $$\lim_{\epsilon \to 0}\limsup_{r \downarrow 0} \frac{\Psi(\epsilon r)}{\Psi(r)\epsilon^2} = 0.$$ 
		Then $\lim_{r \downarrow 0} \frac{\Psi(r)}{r^2}=0$.
	\end{lemma}
	\begin{proof}
		Set
		\[
		h(r):=\frac{\Psi(r)}{r^2}.
		\]
		By assumption, we may choose $\epsilon\in(0,1)$ and $r_0>0$ such that
		\[
		h(\epsilon r)\leq \frac12 h(r)
		\]
		for all $r\in(0,r_0)$. Iterating, for every $n\in\mathbb{N}$,
		\[
		h(\epsilon^n r_0)\leq 2^{-n}h(r_0).
		\]
		
		Now let $r\in(0,r_0)$. Choose $n\in\mathbb{N}$ such that
		\[
		\epsilon^{n+1}r_0<r\leq \epsilon^n r_0.
		\]
		Since $\Psi$ is a scale function,
		\[
		\frac{\Psi(r)}{r^2}
		\lesssim
		\frac{\Psi(\epsilon^n r_0)}{\epsilon^{2n+2}r_0^2}
		=
		\epsilon^{-2}h(\epsilon^n r_0)
		\leq
		\epsilon^{-2}2^{-n}h(r_0).
		\]
		As $r\downarrow0$, we have $n\to\infty$, and hence
		\[
		\lim_{r\downarrow0}\frac{\Psi(r)}{r^2}=0.
		\]
	\end{proof}
	
	\begin{theorem} \label{t:zero-ug}
	 Assume the hypotheses of Theorem~\ref{t:diff-space}.
	 Then every $f\in\LIP(X)$ has minimal $\ast$-upper gradient
	 $g_f=0$ $m$-almost everywhere.
	\end{theorem}
	
	\begin{proof}
		Let $D$ denote the doubling constant of $m$, and let $A,C_{PI}\geq1$ be the constants in Lemma \ref{lem:PIptwise}. Let $C_D>0$ be the constant in Lemma \ref{l:zero-ug}.
		
		By \cite[Corollary 1.10]{Mur-chain}, there exists $C_\Psi\geq1$ such that
		\begin{equation}\label{e:psi-quasi}
			\frac{\Psi(s)}{s^2}
			\leq
			C_\Psi\frac{\Psi(R)}{R^2},
			\qquad 0<s\leq R \le \diam(X,d).
		\end{equation}
		By ${\rm Cap}(\Psi)_{\le}$ and regularity, after changing the structural constant if necessary, there exists $C_U>0$ such that for every $p\in X$ and $r>0$ there exists $\phi\in\mathcal{F} \cap C(X)$, with $0\leq\phi\leq1$, satisfying
		\[
		\phi=1\quad\text{on }B(p,5r/4),\qquad
		\phi=0\quad\text{on }X\setminus B(p,3r/2),
		\]
		and
		\begin{equation}\label{e:cap-cutoff}
			\mathcal{E}(\phi,\phi)
			\leq
			C_U\frac{m(B(p,r))}{\Psi(r)}.
		\end{equation}
		Choose $\delta\in(0,1)$ sufficiently small that
		\begin{equation}\label{e:delta-small}
			\delta<\min\left\{\frac14,\frac{2}{A}\right\},
			\qquad
			C_{\rm PI}
			\left(2C_UC_DD^3C_\Psi\delta\right)^{1/2}
			\leq\frac12.
		\end{equation}
		Using the hypothesis on $\Psi$, we may decrease $\delta$ further so that
		\[
		\limsup_{r\downarrow0}
		\frac{\Psi(\delta r)}{\Psi(r)\delta^2}<1.
		\]
		Hence there exists $r_0>0$ such that
		\begin{equation}\label{e:psi-ub}
			\frac{\Psi(\delta r)}{\Psi(r)\delta^2}
			\leq1
		\end{equation}
		for all $r\in(0,r_0)$.
		
		Suppose, to the contrary, that there exists $f\in\LIP(X)$ such that $g_f\neq0$ on a set of positive measure. By Lemma \ref{l:zero-ug}, there exist $p\in X$ and $r\in(0,r_0)$ such that the annulus
		\[
		B(p,2r)\setminus B(p,r)
		\]
		is $(\delta,\delta/C_D)$-connected.	Let $\phi\in\mathcal{F}$ be the cutoff from \eqref{e:cap-cutoff}, represented by its precise representative, and set
		\[
		\Lambda:=2C_UC_DD^3\delta^{-1}.
		\]
		Define
		\[
		E:=
		\left\{
		x\in B(p,2r):
		M_{2r}\Gamma(\phi,\phi)(x)
		\geq\frac{\Lambda}{\Psi(r)}
		\right\}.
		\]
		By Lemma \ref{lem:weak-maximal}, $E$ is Borel and
		\begin{align*}
			\frac{m(E)}{m(B(p,2r))}
			&\leq
			\frac{D^3\Psi(r)}{\Lambda}
			\frac{\Gamma(\phi,\phi)(B(p,4r))}
			{m(B(p,2r))}\\
			&\leq
			\frac{D^3\Psi(r)}{\Lambda}
			\frac{\mathcal{E}(\phi,\phi)}
			{m(B(p,2r))}\\
			&\leq
			\frac{C_UD^3}{\Lambda}
			\leq
			\frac{\delta}{2C_D}.
		\end{align*}
		In particular,
		\[
		\frac{m(E)}{m(B(p,2r))}
		<
		\frac{\delta}{C_D}.
		\]
		
		By the $(\delta,\delta/C_D)$-connectivity of the annulus, there exists a curve fragment
		\[
		\gamma\in\Gamma(B(p,r),X\setminus B(p,2r))
		\]
		such that
		\[
		\gap(\gamma)\leq\delta r
		\]
		and
		\[
		\gamma(\dom(\gamma))\cap E
		\subset
		\{\gamma(\min(\dom(\gamma)))\}.
		\]
		By truncating $\gamma$ at its first exit from $B(p,2r)$, if necessary, we may assume that
		\[
		\gamma(t)\in B(p,2r)
		\]
		for every $t\in\dom(\gamma)\setminus\{\max(\dom(\gamma))\}$.
		We parametrize $\gamma$ by unit speed.
		Thus, for every
		$t\in\dom(\gamma)\setminus
		\{\min(\dom(\gamma)),\max(\dom(\gamma))\}$,
		\[
		M_{2r}\Gamma(\phi,\phi)(\gamma(t))
		<
		\frac{\Lambda}{\Psi(r)}.
		\]
		In particular, Lemma \ref{lem:PIptwise} implies that every such $\gamma(t)$ is a Lebesgue point of $\phi$.
		
		Let $t,t'$ be interior points of $\dom(\gamma)$ such that
		\[
		d(\gamma(t),\gamma(t'))\leq\delta r.
		\]
		Since $A\delta r<2r$, Lemma \ref{lem:PIptwise} gives, after increasing $C_{PI}$ by a fixed factor if necessary,
		\begin{equation}\label{e:phi-two-point}
			|\phi(\gamma(t))-\phi(\gamma(t'))|
			\leq
			C_{PI}\Lambda^{1/2}
			\frac{\Psi(d(\gamma(t),\gamma(t')))^{1/2}}
			{\Psi(r)^{1/2}}.
		\end{equation}
		If $s\leq\delta r$, then by \eqref{e:psi-quasi} and \eqref{e:psi-ub},
		\begin{align}
			\frac{\Psi(s)}{\Psi(r)}
			&\leq
			C_\Psi
			\frac{s^2}{\delta^2r^2}
			\frac{\Psi(\delta r)}{\Psi(r)}
			\leq
			C_\Psi\frac{s^2}{r^2}.
			\label{e:psi-linear}
		\end{align}
		Hence \eqref{e:phi-two-point} yields
		\[
		|\phi(\gamma(t))-\phi(\gamma(t'))|
		\leq
		C_{PI}(C_\Psi\Lambda)^{1/2}
		\frac{d(\gamma(t),\gamma(t'))}{r}.
		\]
		Together with $0\leq\phi\leq1$ and the fact that $\phi\circ\gamma$ is constant in a neighbourhood of each endpoint of $\dom(\gamma)$, this shows that $\phi\circ\gamma$ is Lipschitz.
		
		Moreover, for every interior $t\in\dom(\gamma)$, letting $t'\to t$ through $\dom(\gamma)$ in \eqref{e:phi-two-point} gives
		\[
		(\phi\circ\gamma)'(t)=0
		\]
		at almost every such $t$, since Lemma \ref{l:sub-gauss} gives
		\[
		\lim_{s\downarrow0}\frac{\Psi(s)^{1/2}}{s}=0.
		\]
		
		Let $\{(a_i,b_i):i\in I\}$ denote the gaps of $\gamma$. Since $\phi\circ\gamma$ is Lipschitz and its derivative vanishes almost everywhere,
		\begin{equation}\label{e:gap-decomp}
			|\phi(\gamma(\max(\dom(\gamma))))
			-\phi(\gamma(\min(\dom(\gamma))))|
			\leq
			\sum_{i\in I}
			|\phi(\gamma(b_i))-\phi(\gamma(a_i))|.
		\end{equation}
		
		If $a_i=\min(\dom(\gamma))$, then
		\[
		d(\gamma(a_i),\gamma(b_i))
		\leq\gap(\gamma)\leq\delta r<r/4.
		\]
		Since $\gamma(a_i)\in B(p,r)$, both endpoints of this gap lie in
		$B(p,5r/4)$, and hence
		\[
		\phi(\gamma(a_i))=\phi(\gamma(b_i))=1.
		\]
		Similarly, if $b_i=\max(\dom(\gamma))$, then
		\[
		d(p,\gamma(a_i))
		\geq
		2r-\delta r>\frac32r,
		\]
		so
		\[
		\phi(\gamma(a_i))=\phi(\gamma(b_i))=0.
		\]
		Thus gaps incident to the endpoints contribute nothing to
		\eqref{e:gap-decomp}.
		
		For every remaining gap, both $\gamma(a_i)$ and $\gamma(b_i)$ avoid $E$, and
		\[
		s_i:=d(\gamma(a_i),\gamma(b_i))
		\leq\gap(\gamma)\leq\delta r.
		\]
		Hence, by \eqref{e:phi-two-point} and \eqref{e:psi-linear},
		\[
		|\phi(\gamma(b_i))-\phi(\gamma(a_i))|
		\leq
		C_{PI}(C_\Psi\Lambda)^{1/2}\frac{s_i}{r}.
		\]
		Summing over the gaps and using
		\[
		\sum_{i\in I}s_i=\gap(\gamma)\leq\delta r,
		\]
		we obtain
		\begin{align*}
			|\phi(\gamma(\max(\dom(\gamma))))
			-\phi(\gamma(\min(\dom(\gamma))))|
			&\leq
			C_{PI}(C_\Psi\Lambda)^{1/2}\delta\\
			&=
			C_{PI}
			\left(2C_UC_DD^3C_\Psi\delta\right)^{1/2}\\
			&\leq\frac12,
		\end{align*}
		where the last inequality follows from \eqref{e:delta-small}.
		On the other hand,
		$\phi(\gamma(\min(\dom(\gamma))))=1$ and $\phi(\gamma(\max(\dom(\gamma))))=0$, 
		so the left hand side above equals $1$, a contradiction. Therefore
		$g_f=0$ $m$-almost everywhere for every $f\in\LIP(X)$.
	\end{proof}

		  The proof of Theorem \ref{t:diff-space} is based on the vanishing of minimal $\ast$-upper gradients obtained from the above contradiction argument. We note, that roughly dual to the notion of minimal $\ast$-upper gradients is the notion of Alberti representations from \cite{Bate}, and a proof of the following theorem could likely be obtained by using the blow-up tools of \cite{cks} with the characterization of differentiability spaces using Alberti representations in \cite{Bate}. This duality is made precise in \cite{bate2024fragment}. This alternative argument seems to take considerably more effort than the proof presented in this section, and is not pursued further. 
	
	\begin{proof}[Proof of Theorem \ref{t:diff-space}]
		Suppose, to the contrary, that $(X,d,m)$ is a Lipschitz
		differentiability space. By \cite[Theorem 1.7]{bate2024fragment}, for
		every $f\in\LIP(X)$ the minimal $\ast$-upper gradient satisfies
		\[
		g_f=\Lip f
		\qquad m\text{-almost everywhere},
		\]
		where $\Lip f$ denotes the pointwise upper Lipschitz constant of $f$.
		By Theorem \ref{t:zero-ug}, $g_f=0$ $m$-almost everywhere for every
		$f\in\LIP(X)$. Hence
		\[
		\Lip f=0
		\qquad m\text{-almost everywhere}
		\]
		for every $f\in\LIP(X)$.
		
		Let $(U,\varphi)$ be a Lipschitz differentiability chart of dimension
		$n\geq1$, where
		\[
		\varphi=(\varphi^1,\ldots,\varphi^n).
		\]
		Fix $j\in\{1,\ldots,n\}$. Since
		\[
		\varphi^j(y)-\varphi^j(x)
		=
		e_j^*(\varphi(y)-\varphi(x)),
		\]
		the coordinate functional $e_j^*$ is a differential of $\varphi^j$ at
		every $x\in U$. On the other hand, $\Lip\varphi^j=0$ for
		$m$-almost every $x\in U$, and hence
		\[
		|\varphi^j(y)-\varphi^j(x)|=o(d(x,y))
		\]
		at almost every such $x$. Thus the zero linear map is also a differential
		of $\varphi^j$ at $x$. Since $e_j^*\neq0$, this contradicts the uniqueness
		of the differential. Therefore $(X,d,m)$ is not a Lipschitz
		differentiability space.
	\end{proof}
	
	\begin{remark} Theorems~\ref{t:diff-space} and~\ref{t:zero-ug} extend,
		with minor modifications, to the nonlinear setting of
		$p$-strongly local regular energy spaces introduced in
		\cite{ErikssonBiqueMuruganEID}.
		We conjecture that the conclusions of Theorems \ref{t:zero-ug} and   \ref{t:diff-space} remain valid under the weaker assumption of $\Psi$ given in Lemma \ref{l:sub-gauss}; that is, 
		\[
		\lim_{r \downarrow 0}  \frac{\Psi(r)}{r^2}=0.
		\]
		The assumption on the scale function $\Psi$ in Theoren \ref{t:diff-space} also plays an important role in the proof of the singularity of energy measures in \cite{KM20}. There is a similar conjecture to prove the singularity of energy measure under the slightly weaker assumption on $\Psi$ \cite[Conjecture 2.15]{KM20}. We refer to \cite[Remark 2.14]{KM20} for examples of scale functions that satisfy the weaker condition on $\Psi$ but not the stronger one assumed in the Theorem \ref{t:diff-space}.
	\end{remark}

\bigskip
\noindent\textsc{Sylvester Eriksson-Bique}\\
Department of Mathematics and Statistics, University of Jyv\"askyl\"a,
P.O. Box 35, FI-40014 University of Jyv\"askyl\"a, Finland\\
\textit{Email:} \href{mailto:sylvester.d.eriksson-bique@jyu.fi}{\texttt{sylvester.d.eriksson-bique@jyu.fi}}

\medskip
\noindent\textsc{Mathav Murugan}\\
Department of Mathematics, University of British Columbia,
Vancouver, BC V6T 1Z2, Canada\\
\textit{Email:} \href{mailto:mathav@math.ubc.ca}{\texttt{mathav@math.ubc.ca}}
\end{document}